\documentclass[a4paper,reqno,12pt]{amsart}
\usepackage{fancyhdr}
\usepackage{amsmath}
\usepackage{dsfont}
\usepackage{hyperref}
\usepackage[mathscr]{eucal}
\usepackage[cp1251]{inputenc}
\usepackage[english]{babel}
\usepackage{cite,enumerate,float,indentfirst}
\usepackage{graphicx}
\usepackage{xcolor}
\usepackage{latexsym,a4,mathrsfs,amsthm,amsmath,amssymb,url}
\usepackage{amsfonts}
\usepackage{amssymb}
\usepackage{epstopdf}

\newcommand\NoBlackBoxes{\global\overfullrule0pt}
\NoBlackBoxes
\numberwithin{equation}{section}

\newtheorem{theorem}{Theorem}[section]
\newtheorem{lemma}[theorem]{Lemma}

\newtheorem{corollary}[theorem]{Corollary}
\theoremstyle{remark}

\newcommand{\R}{\mathbb{R}}
\newcommand{\C}{\mathbb{C}}
\newcommand{\T}{\mathbb{T}}

\newcommand{\J}{\mathbb{J}}
\newcommand{\K}{\mathbb{K}}
\newcommand{\Cond}{{\bf (C0)}}
\newcommand{\CondTwo}{{\bf (C1)}}
\newcommand{\CondThree}{{\bf (CG)}}

\newcommand{\X}{{\bf X}}
\newcommand{\U}{{\bf U}}
\newcommand{\A}{{\bf A}}

\newcommand{\EE}{{\bf E}}

\newcommand{\II}{{\bf I}}
\newcommand{\RR}{{\bf R}}
\newcommand{\RS}{{\bf S}}
\newcommand{\RT}{{\bf T}}
\newcommand{\V}{{\bf V}}
\newcommand{\y}{{\bf y}}

\newcommand{\W}{{\bf W}}

\newcommand{\Y}{{\bf Y}}

\newcommand{\ee}{{\bf e}}
\newcommand{\uu}{{\bf u}}

\DeclareMathOperator{\Tr}{Tr}

\DeclareMathOperator{\E}{\mathbb{E}}

\DeclareMathOperator{\Pb}{\mathbb{P}}
\DeclareMathOperator{\one}{\mathds{1}}
\DeclareMathOperator{\re}{Re}
\DeclareMathOperator{\imag}{Im}

\DeclareMathOperator{\supp}{supp}

\begin{document}

\vspace{1in}

\title[Local Semicircle Law]{\bf On the Local Semicircular Law for Wigner Ensembles}

\author[F. G{\"o}tze]{F. G{\"o}tze}
\address{Friedrich G{\"o}tze, Faculty of Mathematics, Bielefeld University
}
\email{goetze@math.uni-bielefeld.de}

\author[A. Naumov]{A. Naumov}
\address{Alexey A. Naumov, Skoltech Institute of Science and Technology; Moscow State University; IITP RAS.
 }
\email{a.naumov@skoltech.ru}

\author[A. Tikhomirov]{A. Tikhomirov}
\address{Alexander N. Tikhomirov, Department of Mathematics, 
Komi Science Center of Ural Division of RAS; Syktyvkar State University
 }
\email{tikhomirov@dm.komisc.ru}

\author[D. Timushev]{D. Timushev}
\address{Dmitry A. Timushev\\
	Department of Mathematics,
	Komi Science Center of Ural Division of RAS
}
\email{timushev@dm.komisc.ru}

\thanks{All authors were supported by CRC 701 “Spectral Structures and Topological Methods in
Mathematics”. A, Naumov, A. Tikhomirov and D. Timushev were supported by RFBR N~14-01-00500 and by Programs of UD RAS, project N~15-16-1-3. A.~Naumov was also supported by Hong Kong RGC GRF~403513, RFBR N~16-31-00005 and President’s of Russian Federation Grant for young scientists N~4596.2016.1 }

\keywords{Random matrices, local semicircle law, Stieltjes transform, rigidity, delocalization, rate of convergence, mean spectral distribution}

\date{\today}

\begin{abstract}
We consider a random symmetric matrix $\X = [X_{jk}]_{j,k=1}^n$ with upper triangular entries being i.i.d. random variables with mean zero and unit variance. We additionally suppose that $\E |X_{11}|^{4 + \delta} =: \mu_{4+\delta} < \infty$ for some $\delta > 0$. The aim of this paper is to significantly extend a  recent result of the authors~\cite{GotzeNauTikh2015a} and show that with high probability the typical distance between the Stieltjes transform of the empirical spectral distribution (ESD) of the matrix $n^{-\frac{1}{2}} \X$ and Wigner's semicircle law is of order $(nv)^{-1} \log n$, where $v$ denotes the distance  to the real line in the complex plane. We apply this result to the rate of convergence of the ESD to the distribution function of the semicircle law  as well as to  rigidity of eigenvalues and eigenvector delocalization significantly extending a recent result by G\"otze, Naumov and Tikhomirov~\cite{GotzeNauTikh2015b}. The result on delocalization is optimal by comparison  with GOE ensembles. Furthermore the techniques of this paper provide a new shorter proof for the optimal $O(n^{-1})$ rate of convergence of the expected ESD to the semicircle law.
\end{abstract}

\maketitle

\section{Introduction and main result}

Let $\X = [X_{jk}]_{j,k=1}^n$ be a random Hermitian matrix with upper triangular entries being independent random variables with mean zero and unit variance. Denote the $n$ eigenvalues of the symmetric matrix $\W: = \frac{1}{\sqrt n} \X$  in the increasing order by
$$
\lambda_1(\W) \le ... \le \lambda_n(\W)
$$
and introduce the eigenvalue counting function $N_I(\W):= |\{1 \le k \le n: \lambda_k(\W) \in I\}|$
for any interval $I \subset \R$, where $|A|$ denotes the number of elements in the set $A$. The pioneering result of E. Wigner~\cite{Wigner1955} states that for any interval $I \subset \R$ of fixed length and independent of $n$
\begin{equation}\label{eq: wigner's semicircle law global regime}
\lim_{n \rightarrow \infty} \frac{1}{n}\E N_I(\W) = \int_I g_{sc}(\lambda) \, d\lambda,
\end{equation}
where
$g_{sc}(\lambda) := \frac{1}{2\pi} \sqrt{(4 - \lambda^2)_{+}}$ and $(x)_{+}: = \max(x, 0)$.  Wigner considered the special case when all $X_{jk}$ take only two values $\pm 1$ with equal probabilities. Later on the result~\eqref{eq: wigner's semicircle law global regime} was called {\it Wigner's semicircle law} and has been extended in various aspects, see, for example,~\cite{Arnold1967},~\cite{Pastur1973},~\cite{Girko1985uspexi},~\cite{GotTikh2006},~\cite{Naumov2013gaussiancase} and~\cite{GotNauTikh2012prob}. For an extensive list of references we refer to the monographs~\cite{AndersonZeit},~\cite{BaiSilv2010} and~\cite{Tao2012}. In what follows we call {\it Wigner's semicircle law or semicircle law} not only a result of type~\eqref{eq: wigner's semicircle law global regime}, but the limiting probability distribution as well.

If an interval $I$ is of fixed length, independent of $n$, it typically contains a  macroscopically large number of eigenvalues, which means a number of order $n$. In this case
we may rewrite~\eqref{eq: wigner's semicircle law global regime} in the following form
\begin{equation}\label{eq: wigner's semicircle law global regime 2}
\frac{1}{n |I|}\E N_I(\W) = \frac{1}{|I|}\int_I g_{sc}(\lambda) \, d\lambda + o\left(\frac{1}{|I|}\right).
\end{equation}
It is of  interest to investigate the case  of smaller intervals where the number of eigenvalues cease to be macroscopically large. In this case the second term on the right hand side of~\eqref{eq: wigner's semicircle law global regime 2} needs to be refined. An appropriate analytical tool for asymptotic approximations is
the Stieltjes transform of the empirical spectral distribution function $F_n(x): = \frac{1}{n} N_{(-\infty, x]}(\W)$, which is given by
$$
m_n(z) := \int_{-\infty}^\infty \frac{d F_n(\lambda)}{\lambda - z}   = \frac{1}{n}\Tr (\W - z \II)^{-1} = \frac{1}{n} \sum_{j=1}^n \frac{1}{\lambda_j(\W) - z},
$$
where $z = u + i v, v >0$. For the imaginary part of $m_n(z)$ we get
$$
\imag m_n(u + i v) = \int_{-\infty}^\infty \frac{v}{(\lambda - u)^2 + v^2} \, d F_n(\lambda) = \frac{1}{v}\int_{-\infty}^\infty \mathcal P\left(\frac{u-\lambda}{v}\right) \, d F_n(\lambda)
$$
which is a kernel density estimator with a Poisson kernel $\mathcal P(x)$  and bandwidth $v$. For a meaningful estimator of the spectral density we cannot allow the distance $v$ to the real line, that is the bandwidth  of the kernel density estimator, to be smaller than the typical $\frac{1}{n}$ -distance between eigenvalues.
Hence, in what follows we shall be mostly interested in the situations when $v \geq \frac{c}{n}, c > 0$, where in some situations $c$ may depend on $n$, growing  for example like $\log n$.

Under rather general conditions one can establish the convergence  of  $m_n(z)$  for {\it fixed} $v > 0$  to the Stieltjes transform of Wigner's semicircle law which is given by
$$
s(z) = \int_{-\infty}^\infty \frac{g_{sc}(\lambda)\, d\lambda}{\lambda-z} = -\frac{z}{2} + \sqrt{\frac{z^2}{4} - 1}.
$$
It is much more difficult  to establish the convergence in the region $1 \geq v \geq \frac{c}{n}$. Significant progress in that direction was recently made in a series of results by L.~Erd{\"o}s, B.~Schlein, H.-T.~Yau and et al., \cite{ErdosSchleinYau2009},~\cite{ErdosSchleinYau2009b} ,~\cite{ErdosSchleinYau2010},~\cite{ErdKnowYauYin2013}, showing
that with high probability uniformly in $u \in \R$
\begin{equation}\label{fluctuations of m_n around s}
|m_n(u+iv) - s(u+iv)| \le \frac{\log^\beta n}{nv}, \quad \beta > 0,  
\end{equation}
which they called  {\it local semicircle law}. It means  that the fluctuations of $m_n(z)$ around $s(z)$ are of order $(nv)^{-1}$ (up to a logarithmic factor). The value of $\beta$  depends on $n$, more precisely on $\beta: = \beta_n = c\log \log n$, where $c>0$ denotes some constant. In in the papers~\cite{ErdosSchleinYau2009},~\cite{ErdosSchleinYau2009b},~\cite{ErdosSchleinYau2010} the inequality ~\eqref{fluctuations of m_n around s}
has been shown  assuming that the distribution of $X_{jk}$ has has sub-exponential tails
 for all $1 \le j, k \le n$. Moreover in~\cite{ErdKnowYauYin2013} this assumption had been relaxed to requiring $\E |X_{jk}|^p \le \mu_p$ for all $p \geq 1$, where $\mu_p$ are some constants. 

Without loss of generality we will assume in what follows that $\X$ is a real symmetric matrix. We say that the conditions $\Cond$ hold if:
\begin{itemize}
	\item \quad $X_{jk}, 1\le j\le k\le n$, are i.i.d. real random variables;
	\item \quad $\E X_{11} = 0, \, \E X_{11}^2 = 1$;
	\item \quad $\E|X_{11}|^{4+\delta} =: \mu_{4+\delta} < \infty$ for some $\delta > 0$.
\end{itemize}    
Our results proven below apply to the case of Hermitian matrices as well. Here we may additionally assume for simplicity that $\re X_{jk}$ and $\imag X_{jk}$ are independent r.v. for all $1 \le j < k \le n$. 
Otherwise  one needs to extend the moment inequalities for linear and quadratic forms in complex r.v. (see~\cite{GotzeNauTikh2015a}[Theorem A.1-A.2]) to the case of dependent real and imaginary parts, the details of which we omit.

The result~\eqref{fluctuations of m_n around s} under the conditions $\Cond$ was proved in a series of papers~\cite{ErdKnowYauYin2013a},~\cite{ErdKnowYauYin2012},~\cite{LeeYin2014} with an $n$-dependent value $\beta = c\log \log n$. In~\cite{GotzeNauTikh2015a} we gave a self-contained proof based on the methods developed in~\cite{GotTikh2015},~\cite{GotzeTikh2014rateofconv} while at the same time reducing the power of $\log n$ from $\beta = c\log \log n$ to $\beta = 2$. Our work and some crucial bounds of our proof were motivated by the methods used in a recent paper of C.~Cacciapuoti, A.~Maltsev and B.~Schlein,~\cite{Schlein2014}, where the authors improved the log-factor dependence in~\eqref{fluctuations of m_n around s} in the sub-Gaussian case. Let $\kappa: = \frac{\delta}{2(4+\delta)}$. In ~\cite{GotzeNauTikh2015a}[Theorem~1.1] it is shown,  assuming conditions $\Cond$, that for  any fixed $V > 0$ there exist positive constants $A_0, A_1$ and $C$ depending on $\delta$ and $V$ such that
\begin{equation}\label{main result part 1 1}
\E |m_n(z) - s(z)|^p \le \left(\frac{Cp^2}{nv}\right)^p,
\end{equation}
for all $1 \le p \le A_1 (nv)^{\kappa}$, $V \geq v \geq A_0 n^{-1}$ and $|u| \le 2+v$. Note   that when stating  that some constant $C$ depends on $\delta$ we actually mean  that it may depends on $\mu_{4+\delta}$ as well, that is $C = C(\delta, \mu_{4+\delta})$. Applying Markov's inequality we may rewrite \eqref{main result part 1 1} in what follows in the following form
\begin{equation}\label{main result part 1 1 probability}
\Pb \left( |m_n(z) - s(z)| \geq \frac{K}{nv}\right) \le \left(\frac{Cp^2}{K}\right)^p,
\end{equation}
for all $1 \le p \le A_1 (nv)^{\kappa}$, $V \geq v \geq A_0 n^{-1}$ and $|u| \le 2+v$. For applications  the range of $v$, such that~\eqref{main result part 1 1} holds for fixed $p$ is 
relevant. It is clear that $V \geq v \geq C p^\frac{1}{\kappa} n^{-1}$.  
Since we are interested in error probabilities which are polynomially small only, it suffices to choose $p$ to be of order $\log n$, which implies that $V \geq v \geq C n^{-1} \log^\frac{1}{\kappa} n$. At the same time the constant $K$ in~\eqref{main result part 1 1 probability} should be of order $\log^2 n$. Comparing with~\eqref{fluctuations of m_n around s} we get $\beta = 2$.

In the region $|u| > 2 + v$ we can control the imaginary part of $m_n(s)$ only. It was proved in~\cite{GotzeNauTikh2015a}[Theorem~1.1] that for any $u_0 > 0$ there exist positive constants $A_0, A_1$ and $C$ depending on $u_0, V$ and $\delta$ such that
\begin{equation}\label{main result part 1 2}
\E |\imag m_n(z) - \imag s(z)|^p \le \left(\frac{Cp^2}{nv}\right)^p,
\end{equation}
for all $1 \le p \le A_1 (nv)^\kappa$, $V \geq v \geq A_0 n^{-1}$ and $|u| \le u_0$.

As mentioned above, we are  interested  in the case when $p$ is of order $\log n$ which implies $V \geq v \geq  n^{-1} \log^\frac{1}{\kappa} n$. This choice yields that in our applications  all bounds will depend on $\log^\frac{1}{\kappa} n$. The power $\kappa^{-1}$, which is independent of $n$,  may be rather large for $\delta$ near zero. The aim of the current paper is to strengthen the results of~\cite{GotzeNauTikh2015a},~\cite{GotzeNauTikh2015b} by proving  bounds of type~\eqref{fluctuations of m_n around s} with $\beta = 1$ while  at the same time showing that $p$ and $v$ may be taken of order $\log n$ and $n^{-1} \log n$ respectively. This is done in Theorem~\ref{th:main} below. To this end  we combine our techniques from~\cite{GotzeNauTikh2015a} with fruitful ideas from~\cite{ErdKnowYauYin2012} and~\cite{LeeYin2014} in particular their  moment matching technique. A crucial  result in that direction is Lemma~\ref{main lemma}. See the sketch of proof below in Section~\ref{Sketch of the proof}. It still remains one challenging open problem, namely extending the bounds to weaken the moment condition to $\delta = 0$.

\subsection{Main result} Let us introduce the following notation
$$
\Lambda_n(z) := m_n(z) - s(z), \quad z = u + i v.
$$
The main result of this paper is the following theorem, which estimates the fluctuations~\eqref{fluctuations of m_n around s} and strengthens~\eqref{main result part 1 1} and~\eqref{main result part 1 2}. 
\begin{theorem}\label{th:main}
Assume that the conditions $\Cond$ hold and let $V > 0$ be some constant. \\
$(i)$ There exist positive constants $A_0, A_1$ and $C$ depending on $V$ and $\delta$  such that
$$
\E |\Lambda_n(z)|^p \le \left(\frac{Cp}{nv}\right)^p,
$$
for all $1 \le p \le A_1 \log n$, $V \geq v \geq A_0 n^{-1} \log n$ and $|u| \le 2+v$. \\

\noindent
$(ii)$ For any $u_0 > 0$ there exist positive constants $A_0, A_1$ and $C$ depending on $u_0, V$ and $\delta$
such that
$$
\E |\imag \Lambda_n(z)|^p \le \left(\frac{Cp}{nv}\right)^p,
$$
for all $1 \le p \le A_1 \log n$, $V \geq v \geq A_0 n^{-1} \log n$ and $|u| \le u_0$.
\end{theorem}

As a consequence of this result we may show that similarly to~\eqref{main result part 1 1 probability} for all $K > 0$
\begin{equation}\label{fluctuations of m_n around s 2}
\Pb\left(|\Lambda_n(z)| \geq \frac{K}{nv}\right) \le \left(\frac{C p}{K}\right)^p,
\end{equation}
valid for all $1 \le p \le A_1 \log n$, $V \geq v \geq A_0 n^{-1} \log n$ and $|u| \le 2+v$. Taking $p$ and $K$ of order $\log n$ we may guarantee that~\eqref{fluctuations of m_n around s 2} is less then, for example, $n^{-2}$. Thus, comparing with \eqref{fluctuations of m_n around s}, we get $\beta = 1$.

Applications of Theorem~\ref{th:main} outside the limit spectral interval, that is for $|u| \geq 2$,  require  stronger bounds on $\imag \Lambda_n$. Let us denote
\begin{equation}\label{eq: def gamma}
\gamma: = \gamma(u):= ||u| - 2|
\end{equation}
and introduce the following quantity 
$$
\alpha: = \alpha(\delta) = \frac{2}{4+\delta},
$$
which will control the level of truncation of the matrix entries. We say that the set of conditions $\CondTwo$ holds if:
\begin{itemize}
	\item \quad $\Cond$ are satisfied;
	\item \quad $|X_{jk}| \le D n^\alpha, 1 \le j,k \le n$, where $D: =D(\alpha)$ denotes some positive constant depending on $\alpha$ only.
\end{itemize}
By definition $\kappa=\delta/(2 (4+\delta))$, hence we may write $\kappa = \frac{1-2\alpha}{2}$.
\begin{theorem}\label{th: stronger bound for imag part}
Assume that the conditions $\CondTwo$ hold and $u_0 > 2$ and $V > 0$. There exist positive constants $A_0, A_1$ and $C$ depending on $\delta, u_0$ and $V$ such that
$$
\E |\imag \Lambda_n(z)|^p \le \frac{C^p p^p }{n^p (\gamma + v)^p} +  \frac{C^p p^{2p}}{(nv)^{2p} (\gamma + v)^\frac{p}{2}}  + \frac{C^p}{n^p v^\frac{p}{2} (\gamma + v)^\frac{p}{2}} + \frac{C^p p^\frac{p}{2}}{(nv)^\frac{3p}{2} (\gamma + v)^\frac{p}{4} },
$$
for all $1 \le p \le A_1 (nv)^{\kappa}$, $V \geq v \geq A_0 n^{-1}$ and $2 \le |u| \le u_0$.
\end{theorem}

\subsection{Sketch of the proof of Theorem~\ref{th:main} and Theorem~\ref{th: stronger bound for imag part}}\label{Sketch of the proof} The proof of Theorem~\ref{th:main} is similar to the proof~\cite{GotzeNauTikh2015a}[Theorem~1.1]. Applying Lemma~B.1 in~\cite{GotzeNauTikh2015a} see~\eqref{eq: abs value lambda main result section} below and~\cite{Schlein2014}[Proposition~2.2]) it is shown in Section~\ref{proof of the main result} that one may estimate $\E |\Lambda_n(z)|^p$ or $\E |\imag \Lambda_n(z)|^p$ (depending on  $\re(z)$ being near  
or far from the spectral interval $[-2,2]$) 
by the moments $\E |T_n(z)|^p$ (see definition~\eqref{definition of T}). 

To estimate $\E |T_n(z)|^p$ we may repeat all the steps of the proof of Theorem~2.1 in~\cite{GotzeNauTikh2015a} with  one important modification. One of the crucial steps in the proof is the bound $\E |\RR_{jj}(z)|^p \le C_0^p$. It was shown in~\cite{GotzeNauTikh2015a}[Lemma~4.1] that this bound holds for all $V \geq v \geq  A_0 n^{-1}$ and $1 \le p \le A_1 (nv)^\kappa$. Since we have to choose $p$  of order $\log n$ we need to prove bounds in the region $v \geq \tilde v : = n^{-1}\log^\frac{1}{\kappa} n$. In order to close the gap in $v$ from $\tilde v$ to $v_0: = n^{-1} \log n$ with  $p$ being still of order $\log n$ we apply the following strategy. We start from the fact that $\E |\RR_{jj}(z)|^p \le C_0^p \log^{\left(\frac{1}{\kappa}-1\right) p} n$ for all $v \geq v_0$, see Lemma~\ref{lemma: descent}. In order to remove the logarithmic factor from the r.h.s of the last bound we apply ideas  motivated by moment matching techniques used in~\cite{ErdKnowYauYin2012} and~\cite{LeeYin2014}. That is we construct a symmetric random matrix $\Y:= [Y_{jk}]_{jk}^n$ such that $\E X_{jk}^s = \E Y_{jk}^s, 1\le s\le 4$ and $Y_{jk}$ are sub-gaussian random variables, see Lemma~\ref{lem: random variables} (see Lemma~5.1~\cite{LeeYin2014}). Then we show in Lemma~\ref{lem: bound in the optimal region}   that the bound  $\E |\RR_{jj}(z)|^p \le C^p + \E |\RR_{jj}^\y(z)|^p$ still holds for all $v \geq v_0$ and $p$ of order $\log n$, where $\RR^\y$ denotes  $\RR$ with the $\X$ matrix being replaced by the $\Y$ matrix.
Our technique in the proof of Lemma~\ref{lem: bound in the optimal region} is a Stein type method. Finally, using the sub-gaussian properties of $\Y$ we show in Lemma~\ref{main lemma gauss} that $\E |\RR_{jj}^\y(z)|^p \le C_0^p$ for all $v \geq v_0$ and $1 \le p \le c \log n$, for some constant $c$. All these arguments rely on the proof of Lemma~\ref{main lemma}. Note in contrast that in~\cite{LeeYin2014} the moment matching techniques were used to estimate the distance $m_n(z) - s(z)$ directly combined with a combinatorial approach (see, for example, Lemma~5.1 and Lemma~5.2 in~\cite{LeeYin2014}).

The proof of Theorem~\ref{th: stronger bound for imag part} is based on the same arguments as the proof of Theorem~1.2 in~\cite{GotzeNauTikh2015a}. Note that $\E |T_n(z)|^p$ is bounded in terms of $\E \imag^p \RR_{jj}$. The same arguments as for $\E |\RR_{jj}(z)|^p $ will give us the bound $\E \imag^p \RR_{jj} \le H_0^p \imag^p s(z) + H_0^p p^p (nv)^{-p}$ valid for some big constant $H_0 > 0$ and $v \geq v_0, 1 \le p \le c \log n$. Since we can derive explicit bounds for $\imag s(z)$ inside as well as outside the limit spectrum  we
are able to control the size of $\E |T_n(z)|^q$ as well as of $\E |\imag \Lambda_n(z)|^p$  on the whole real line  in terms of the quantity  $\gamma$ (see~\eqref{eq: def gamma}). This is a another key argument for the proof of Theorem~\ref{th: stronger bound for imag part}.

\subsection{Delocalization, rigidity and rate of convergence}\label{sec: applications} 
In this section we present results about delocalization of eigenvectors, rate of convergence of empirical spectral distribution function (ESD) to the semicircle law and rigidity of eigenvalues. These results strengthen the corresponding results in~\cite{GotzeNauTikh2015b}.

We start this section by showing delocalization  of  eigenvectors. This question has been intensively studied in many papers, for example, in~\cite{ErdosSchleinYau2009}~\cite{GotzeTikh2011rateofconv},~\cite{ErdKnowYauYin2013a} and~\cite{ErdKnowYauYin2012}.

Let us denote by $\uu_j := (u_{j 1}, ... , u_{j n})$ the eigenvectors of $\W$ corresponding to the eigenvalue $\lambda_j(\W)$ .
\begin{theorem}\label{th: delocalization}
	Assume that  conditions $\Cond$ hold with  $\delta = 4$. For any $0 < \phi < 2$ there exist positive constants $C$ and $C_1$ depending on $\phi$ and $\mu_8$ such that for any 
	$$
	\Pb \left(\max_{1 \le j, k \le n} |u_{jk}|^2 \geq \frac{C_1 \log n}{n} \right) \le  \frac{C}{n^{2-\phi}}.
	$$
\end{theorem}
We remark here that it is possible to relax the moment conditions to  the case $0 < \delta < 4$ as well. But here  we may only conclude that there exists some constant $c(\delta) > 0$ depending on $\delta$ such that
$$
\Pb \left(\max_{1 \le j, k \le n} |u_{jk}|^2 \geq \frac{C_1 \log n}{n} \right) \le  \frac{C}{n^{c(\delta)}}.
$$
A comparison with a similar result for the GOE ensemble (see~\cite{AndersonZeit}[Corollary 2.5.4]) 
and the delocalization of eigenvectors  of the unit sphere shows that this result is optimal with respect to  the power of logarithm. It is not clear though whether it is still possible to strengthen the probability bounds above. The numerical calculations in Section~5 of~\cite{GotzeNauTikh2015b} strongly suggest that the actual probability bounds should be very poor. The proof of this theorem is similar to the proof of Theorem~1.4 in~\cite{GotzeNauTikh2015b}, but since this result is optimal we present it here. It is given in Section~\ref{sec: delocalization}. The proof is based on Lemma~\ref{main lemma}. 

The results on rate of convergence and the rigidity of eigenvalues are based on Theorem~\ref{th:main} and Theorem~\ref{th: stronger bound for imag part}. We first investigate the rate of convergence in probability of ESD to the distribution function of Wigner's semicircle law defined as follows $G_{sc}(x): = \int_{-\infty}^x g_{sc}(\lambda) \, d\lambda$. To measure the distance between distribution functions we introduce the uniform distance
$$
\Delta_n^{*}: = \sup_{x \in \R} |F_n(x) - G_{sc}(x)|.
$$
At this point we omit a detailed discussion of  previous results and refer the reader instead to~\cite{GotzeNauTikh2015b}, which provides   links to the  related papers~\cite{GotTikh2003},~\cite{BaiHuPanZhou2011},~\cite{GotzeTikh2011rateofconv},~\cite{GotzeTikh2014rateofconv} and~\cite{TaoVu2013}. We prove the following theorem.
\begin{theorem} \label{th: rate of convergence}
Assume that the conditions $\Cond$ hold. Then there exist positive constants  $c$ and $C$ depending on $\delta$ such that for all $1 \le p \le c \log n$
$$
\Pb\left(\Delta_n^{*} \geq K \right) \le \frac{C^p \log^{2p} n}{K^p n^p}
$$
for all $K > 0$.
\end{theorem}
Let $N[x - \frac{\xi}{2n}, x + \frac{\xi}{2n} ]: = N_I(\W)$ for $I =  [x - \frac{\xi}{2n}, x + \frac{\xi}{2n} ], \xi > 0$. The following result is the direct corollary of Theorem~\ref{th: rate of convergence}.
\begin{corollary}
Assume that  conditions $\Cond$ hold. Then there exist positive constants  $c$ and $C$ depending on $\delta$ such that for all $1 \le p \le c\log n$ and all $\xi > 0, K > 0$
\begin{equation}\label{rigidity}
\Pb\left(\left | \frac{N[x - \frac{\xi}{2n}, x + \frac{\xi}{2n} ]}{\xi} - g_{sc}(x) \right | \geq \frac{K}{\xi} \right) \le \frac{C^p \log^{2p} n}{K^p n^p}.
\end{equation}
\end{corollary}
Taking $K = C_1 n^{-1}\log^2 n$ with $C_1 = C e^{2/c}$ (one may of course take a larger constant)  we get that
\begin{equation}\label{eq: rate of convergence optimal}
\Pb\left(\Delta_n^{*} \geq \frac{C_1 \log^2 n}{n} \right) \le  \frac{1}{n^2}.
\end{equation}
The proof of Theorem~\ref{th: rate of convergence} will be given in Section~\ref{sec: rate of convergence}. We believe that it is still possible to reduce the power of $\log$ in \eqref{rigidity} from $2p$ to $p$ or even $\frac{p}{2}$, which would be optimal  due to the result of Gustavsson~\cite{Gustavsson2005} for the  {\it Gaussian Unitary Ensembles} (GUE). 

Instead of $\Delta_n^{*}$ one may study the following distance of the {\it mean} spectral distribution to its limit
$$
\Delta_n : = \sup_{x \in \R} |\E F_n(x) - G_{sc}(x)|.
$$
The first estimate of $\Delta_n$ was obtained by Z. Bai~\cite{Bai1993}, who showed  the bound $\Delta_n = O(n^{-\frac14})$ assuming $\mu_4<\infty$. Already in 1998 Girko~\cite{Girko1998} published an error bound  order $O(n^{-\frac12})$ under the same moment conditions. In 2002 \cite{Girko2002} he closed gaps in his proof. The same result was independently obtained by Bai and et. al.~\cite{BaiMiaTsay2002} and G\"otze, Tikhomirov~\cite{GotTikh2003}. Girko~\cite{Girko1998}  claimed that the actual rate of convergence of the expected spectral distribution function to the semi-circle law is $O(n^{-1/2})$ even in the Gaussian case. In 2002 G\"otze and Tikhomirov~\cite{GotzTikh2002} showed that for the GUE the rate of convergence to the semi-circle law is $O(n^{-2/3})$. In 2005 \cite{GotzTikh2005} they improved this bound to $O(n^{-1})$. In 2007 in~\cite{TimushevTikhKhol2007} the bound $\Delta_n=O(n^{-1})$ was obtained for the {\it Gaussian Orthogonal Ensemble} (GOE) as well.  Up to this point is was not clear what the optimal rate of convergence to the semi-circular law under weak moment conditions only should be. It follows from~\cite{BobGotTikh2010} that if the distributions of the matrix entries satisfy a Poincar\'e type inequality then $\Delta_n=O(n^{-\frac23})$. Recently G\"otze and Tikhomirov~\cite{GotTikh2015} proved the bound $\Delta_n=O(n^{-1})$ assuming that $\mu_8<\infty$ or $\mu_4<\infty$ combined with the assumption  $|X_{jk}|\le Cn^{\frac14}$ a.s. Finally in~\cite{GotTikh2015weak} the following theorem was proved
\begin{theorem}\label{th: rate of conv expected}
Assume that the conditions $\Cond$ hold. There exist a positive constant $C(\delta)$ depending on $\delta$ such that
$$
\Delta_n \le \frac{C(\delta)}{n}.
$$
\end{theorem}
In Section~\ref{sec: rate of convergence} we give an alternative proof of this theorem, based on the methods developed in this paper. 

Another application of Theorem~\ref{th:main}  is the following result which shows the rigidity of eigenvalues. Let us define the  quantile  position of the $j$-th eigenvalue by
$$
\gamma_j: \quad \int_{-\infty}^{\gamma_j} g_{sc}(\lambda) \, d\lambda = \frac{j}{n}, \quad 1 \le j \le n.
$$
We will prove the following theorem.
\begin{theorem}\label{th: rigidity}
Assume that the conditions $\Cond$ hold and let $K>0$ be an integer. Then \\
\noindent (i). For all $j \in [K, n - K+1]$ there exist constants $c$ and $C, C_1$ depending on $\delta$ such that for  all $1 \le p \le c \log n$ we have
$$
\Pb(|\lambda_j - \gamma_j| \geq C_1 K [\min(j, n- j+1)]^{-\frac13} n^{-\frac23} ) \le \frac{C^p \log^{2p} n}{K^p}.
$$
\noindent (ii). Assume that $\delta = 4$. For any $0 < \phi < 2$ and all $j \leq K$ or $j \geq n - K + 1$ there exist constants $c$ and $C, C_1$ depending on $\phi$ and $\mu_8$ such that for $5 \le p \le c \log n$ 
$$
\Pb(|\lambda_j - \gamma_j| \geq C_1 K [\min(j, n- j+1)]^{-\frac13} n^{-\frac23} ) \le \frac{C}{n^{2-\phi}} + \frac{C^p \log^{12p} n}{K^p}.
$$
\end{theorem}
For comparisons we  refer the interested reader to relevant  results of Gustavsson~\cite{Gustavsson2005} for the Gaussian case, as well as to the results in the papers ~\cite{ErdKnowYauYin2013}[Theorem~7.6],~\cite{ErdKnowYauYin2013a}[Theorem~2.13],~\cite{GotzeTikh2011rateofconv}[Remark~1.2],~\cite{LeeYin2014}[Theorem~3.6] and~\cite{Schlein2014}[Theorem~4] already mentioned above.
The proof of Theorem~\ref{th: rigidity} is similar to the proof of Theorem~1.3 in~\cite{GotzeNauTikh2015b} up to some small improvements due to an improved bound in Theorem~\ref{th: rate of convergence}. For proving the part $(ii)$ we use the result of Theorem~\ref{th: stronger bound for imag part}, using  ideas from\cite{Schlein2014}[Lemma~8.1] and~\cite{ErdKnowYauYin2013}[Theorem~7.6]. Note here that our techniques allow to treat the case $0 < \delta < 4$ in $(ii)$ as well, but with weaker probability bounds in $(ii)$ of order $n^{-c(\delta)}$,
where  $c(\delta) > 0$ depends on moment exponent $\delta$.  
 We omit the details and the proof of Theorem~\ref{th: rigidity}.

\subsection{Notations} \label{sec: notation}
Throughout the paper we will use the following notations. We assume that all random variables are defined on common probability space $(\Omega, \mathcal F, \Pb)$ and let $\E$ be the mathematical expectation with respect to $\Pb$. We denote by $\one[A]$ the indicator function of the set $A$.

We denote by $\R$ and $\C$ the set of all real and complex numbers. We also define $\C^{+}: = \{z \in \C: \imag  z \geq 0\}$. Let $\T = [1, ... , n]$ denotes the set of the first $n$ positive integers. For any $\J \subset \T$ introduce $\T_{\J}: = \T \setminus \J$. To simplify all notations we will write $\T_j, \T_{\J, j}$ instead of $\T_{\{j\}}$ and $\T_{\J \cup \{j\}}$ respectively.

For any matrix $\W$ together with its resolvent $\RR$ and Stieltjes transform $m_n$
we shall systematically use the corresponding notations	$\W^{(\J)}, \RR^{(\J)}, m_n^{(\J)}$, respectively,  for the  sub-matrix  of $\W$ with entries $X_{jk}, j,k \in \T\setminus \J$. For simplicity we write $\W^{(j)}, \W^{(\J,j)}$ instead of $\W^{(\{j\})}, \W^{(\J \cup \{j\})}$. The same is applies to $\RR, m_n$ etc.

By $C$ and $c$ we denote some positive constants. If we write that $C$ depends on $\delta$ we mean that $C = C(\delta, \mu_{4+\delta})$.

For an arbitrary matrix $\A$ taking values in $\C^{n \times n}$ we define the operator norm by $\|\A\|: = \sup_{x \in \R^n: \|x\| = 1} \|\A x\|_2$, where $\|x\|_2 : = (\sum_{j = 1}^n |x_j|^2)^\frac12$. We also define the Hilbert-Schmidt norm by $\|\A\|_2: = \Tr^\frac12 \A \A^{*} = (\sum_{j,k = 1}^n |\A_{jk}|^2)^\frac12$.

\section{Proof of the main result}\label{proof of the main result}
The proofs of Theorem~\ref{th:main} and Theorem~\ref{th: stronger bound for imag part} repeat the arguments of~\cite{GotzeNauTikh2015a}, but for the readers convenience we provide the main steps here. We start  with the recursive representation of the diagonal entries $\RR_{jj}(z) = [(\W - z \II)^{-1}]_{jj}$ of the resolvent.  As noted before we shall systematically use for any matrix $\W$ together with its resolvent $\RR$, Stieltjes transform $m_n$ and etc. the corresponding quantities $\W^{(\J)}, \RR^{(\J)}, m_n^{(\J)}$ and etc. for the corresponding sub matrix with entries $X_{jk}, j,k \in \T\setminus \J$. We will often omit the argument $z$ from $\RR(z)$ and write $\RR$ instead. We may express $\RR_{jj}$ in the following way
\begin{equation}\label{eq: R_jj representation 0}
\RR_{jj} = \frac{1}{-z + \frac{X_{jj}}{\sqrt n} - \frac{1}{n}\sum_{l,k \in \T_j} X_{jk} X_{jl} \RR_{kl}^{(j)}}.
\end{equation}
Let $\varepsilon_j : = \varepsilon_{1j} + \varepsilon_{2j}+\varepsilon_{3j}+\varepsilon_{4j}$, where
\begin{align*}
&\varepsilon_{1j} =  \frac{1}{\sqrt n}X_{jj}, \quad \varepsilon_{2j} = -\frac{1}{n}\sum_{l \ne k \in \T_j} X_{jk} X_{jl} \RR_{kl}^{(j)}, \quad \varepsilon_{3j} = -\frac{1}{n}\sum_{k \in \T_j} (X_{jk}^2 -1) \RR_{kk}^{(j)}, \\
&\varepsilon_{4j}= \frac{1}{n} (\Tr \RR - \Tr \RR^{(j)}).
\end{align*}
Using these notations we may rewrite~\eqref{eq: R_jj representation 0} as follows
\begin{equation}\label{eq: R_jj representation}
\RR_{jj} = -\frac{1}{z + m_n(z)} + \frac{1}{z + m_n(z)} \varepsilon_j \RR_{jj}.
\end{equation}
Summing  last equations for $j = 1, \ldots, n$, we obtain
\begin{equation}\label{eq: m_n equation}
1 + z m_n(z) + m_n^2(z) = T_n,
\end{equation}
where
\begin{equation}\label{definition of T}
T_n: = \frac{1}{n} \sum_{j=1}^n \varepsilon_j \RR_{jj},
\end{equation}
It is well known that $s(z)$ satisfies the following quadratic equation
\begin{equation}\label{eq: s(z) equation}
1 + z s(z) + s^2(z) = 0.
\end{equation}
From~\eqref{eq: m_n equation} and \eqref{eq: s(z) equation} we conclude that
$$
\Lambda_n = \frac{T_n}{z + m_n(z) + s(z)} = \frac{T_n}{b_n(z)},
$$
where
\begin{equation}\label{eq:Lambda_b_bn}	
b_n(z) := b(z) + \Lambda_n, \quad b(z): = z + 2 s(z).
\end{equation}
From Lemma~B.1 in~\cite{GotzeNauTikh2015a} (see also~\cite{Schlein2014}Proposition 2.2])  it follows that for all $v > 0$ and $|u| \le 2+v$
(using the quantities \eqref{eq:Lambda_b_bn})
\begin{equation}\label{eq: abs value lambda main result section}
|\Lambda_n| \le C \min\left\{\frac{|T_n|}{|b(z)|}, \sqrt{|T_n|}\right\}.
\end{equation}
Moreover, let $u_0$ be an arbitrary fixed positive number. Then for all $v>0$ and $|u| \le u_0$
\begin{equation}\label{eq: abs imag lambda main result section}
|\imag\Lambda_n| \le C \min\left\{\frac{|T_n|}{|b(z)|}, \sqrt{|T_n|}\right\}.
\end{equation}
This means that in order to bound $\E |\Lambda_n|^p$ (or $\E|\imag \Lambda|^p$ respectively) it is enough to estimate $\E |T_n|^p$.

Let $V$ be an arbitrary fixed positive real number and $A_0$ is some large constant defined below. We introduce the following region in the complex plane:
\begin{equation}\label{eq: definition reginon D}
\mathbb D: = \{z = u+iv \in \C: |u| \le u_0, V \geq v \geq v_0: = A_0 n^{-1} \log n  \}.
\end{equation}

The following theorem provides a general bound for $\E |T_n|^p$ for all $z \in \mathbb D$ in terms of diagonal resolvent entries. To formulate the result of the theorem we need to introduce additional notations. Let
\begin{equation}\label{definition of A}
\mathcal A(q): = \max_{|\J| \le 1} \max_{j \in \T_\J} \E^{\frac{1}{q}} \imag^q \RR^{(\J)}_{jj},
\end{equation}
where $\J$ may be an empty set or one point set. We also denote
\begin{equation}\label{eq: definition of E}
\mathcal E_p: = \frac{p^p\mathcal A^p(\kappa p)}{(nv)^p}   +  \frac{ p^{2p}}{(nv)^{2p}}  + \frac{ |b(z)|^\frac{p}{2}\mathcal A^\frac{p}{2}(\kappa p)}{(nv)^p },
\end{equation}
where $\kappa = \frac{16}{1 - 2\alpha}$.
\begin{theorem}\label{th: general bound}
Assume that the conditions $\CondTwo$ hold and $u_0 > 2$ and $V > 0$. There exist  positive constants $A_0, A_1$ and $C$
depending on  $\alpha, u_0$ and $V$ such that for all $z \in \mathbb D$ we have
\begin{equation}\label{eq: general bound}
\E|T_n|^p \le C^p \mathcal E_p,
\end{equation}
where $1 \le p \le A_1 \log n$.
\end{theorem}
\begin{proof}
See~\cite{GotzeNauTikh2015a}[Theorem~2.1].
\end{proof}

\begin{proof}[Proof of Theorem~\ref{th:main}] By  standard truncation arguments (see Lemmas~D.1,~D.2,~D.3 in~\cite{GotzeNauTikh2015a}) in what follows  we may assume that
$$
|X_{jk}| \le D n^\alpha \text{ for all }  1 \le j,k \le n
$$
and some $D: = D(\alpha) > 0$.
		
Applying Theorem~\ref{th: general bound} we will show in  section~\ref{sec: imagionary part}, Lemma~\ref{main lemma}, that there exist constants $H_0$ depending on $u_0, V$ and $A_0, A_1$ depending on $\alpha$ and $H_0$ such that
\begin{equation}\label{eq: bound for A (alpha p)}
\mathcal A^{p}(\kappa p) \le H_0^p \imag^p s(z) + \frac{H_0^p p^{p}}{(nv)^p}.
\end{equation}
for all $1 \le p \le A_1 \log n$ and $z \in \mathbb D$.  This inequality and Theorem~\ref{th: general bound} together imply that
\begin{align}\label{eq: T n 5 step}
\E|T_n|^p &\le \frac{C^p p^p \imag^p s(z)}{(nv)^p} +  \frac{C^p p^{2p}}{(nv)^{2p}}  + \frac{C^p |b(z)|^\frac{p}{2} \imag^\frac{p}{2}s(z)}{(nv)^p } + \frac{C^p |b(z)|^\frac{p}{2} p^\frac{p}{2}}{(nv)^\frac{3p}{2} } .
\end{align}
with some new constant $C$ which depends on $H_0$. To estimate $\E|\imag \Lambda_n|^p$ we may choose one of the bounds~\eqref{eq: abs imag lambda main result section}, depending on  $z$ being  near the edge of the limiting  spectral interval or not. If $|b(z)|^p \geq \frac{C^p p^p}{(nv)^p}$ then we may use the bound
$$
\E|\imag \Lambda_n|^p \le \frac{C^p \E |T_n|^p}{|b(z)|^p}.
$$
The r.h.s. of the last inequality may be estimated applying~\eqref{eq: T n 5 step}.  We get
\begin{align*}
\E|\imag \Lambda_n|^p \le\frac{C^p p^p \imag^p s(z)}{(nv)^p |b(z)|^p} +  \frac{C^p p^{2p}}{(nv)^{2p}|b(z)|^p}  + \frac{C^p  \imag^\frac{p}{2}s(z)}{(nv)^p |b(z)|^\frac{p}{2}} + \frac{C^p p^\frac{p}{2}}{(nv)^\frac{3p}{2} |b(z)|^\frac{p}{2} }.
\end{align*}
Since $|b(z)|^p \geq \frac{C^p p^p}{(nv)^p}$ the last inequality may be rewritten in the following way
\begin{align*}
\E|\imag \Lambda_n|^p \le\frac{C^p p^p \imag^p s(z)}{(nv)^p |b(z)|^p}  + \frac{C^p  \imag^\frac{p}{2}s(z)}{(nv)^p |b(z)|^\frac{p}{2}} + \frac{C^p p^{p} }{(nv)^p}.
\end{align*}
It remains to estimate the imaginary part of $s(z)$. Since
$$
\imag^p s(z) \le c^p|b(z)|^p \text { for } |u| \le 2  \text { and }  \imag^p s(z) \le \frac{c^p v^p}{|b(z)|^p} \text{ otherwise}
$$
both inequalities combined yield
\begin{equation}\label{eq: bound for Lambda}
\E|\imag \Lambda_n|^p \le \left(\frac{Cp}{nv}\right)^p,
\end{equation}
where we have used as well the fact that $c\sqrt{\gamma + v} \le |b(z)| \le C\sqrt{\gamma + v} $ for all $|u| \le u_0$, $0 < v \le v_1$. We assume now that $|b(z)|^p  \leq \frac{C^p p^{p}}{(nv)^p}$. Then $\imag^p s(z) \leq \frac{C^p p^{p}}{(nv)^p}$ and we obtain a bound proportional to $|T_n|^\frac12$. Hence,
$$
\E |\imag \Lambda_n|^p \le C^p\E|T_n|^\frac{p}{2} \le \left(\frac{C p}{nv}\right)^p.
$$
Similar arguments apply to $\E|\Lambda_n|^p$.
\end{proof}

\begin{proof}[Proof of Theorem~\ref{th: stronger bound for imag part}] From Theorem~\ref{th: general bound} we may  conclude that
\begin{align}\label{eq: T n 6 step}
\E|T_n|^p &\le \frac{C^p p^p \imag^p s(z)}{(nv)^p} +  \frac{C^p p^{2p}}{(nv)^{3p}}  + \frac{C^p |b(z)|^\frac{p}{2} \imag^\frac{p}{2}s(z)}{(nv)^p }
+ \frac{C^p |b(z)|^\frac{p}{2} p^\frac{p}{2}}{(nv)^\frac{3p}{2} } .
\end{align}
Applying~\eqref{eq: abs imag lambda main result section} we get
$$
\E|\imag \Lambda_n|^p \le \frac{\E |T_n|^p}{|b(z)|^p}.
$$
This inequality together with~\eqref{eq: T n 5 step} leads to
\begin{align}\label{eq: Lambda_n stronger bound}
\E|\imag \Lambda_n|^p \le\frac{C^p p^p \imag^p s(z)}{(nv)^p |b(z)|^p} +  \frac{C^p p^{2p}}{(nv)^{2p}|b(z)|^p}  + \frac{C^p  \imag^\frac{p}{2}s(z)}{(nv)^p |b(z)|^\frac{p}{2}} + \frac{C^p p^\frac{p}{2}}{(nv)^\frac{3p}{2} |b(z)|^\frac{p}{2} }.
\end{align}
Since $c\sqrt{\gamma + v} \le |b(z)| \le C\sqrt{\gamma + v} $ for all $|u| \le u_0$, $0 < v \le v_1$ and
$$
\frac{c v}{\sqrt{\gamma + v}} \le \imag s(z) \le \frac{c v}{\sqrt{\gamma + v}} \quad \text{ for all } 2 \le |u| \le u_0, 0 < v \le v_1,
$$
we finally get
\begin{align}\label{eq: Lambda_n stronger bound 2}
\E|\imag \Lambda_n|^p \le \frac{C^p p^p }{n^p (\gamma + v)^p} +  \frac{C^p p^{2p}}{(nv)^{2p} (\gamma + v)^\frac{p}{2}}  + \frac{C^p}{n^p v^\frac{p}{2} (\gamma + v)^\frac{p}{2}} + \frac{C^p p^\frac{p}{2}}{(nv)^\frac{3p}{2} (\gamma + v)^\frac{p}{4} }.
\end{align}
This bound concludes the proof of the theorem.
\end{proof}

\section{A moment bound for diagonal entries of the resolvent}\label{sec: imagionary part}
In this section we prove bounds for the diagonal entries of the resolvent. As mentioned before in the Introduction and Section~\ref{Sketch of the proof} these bounds will play a crucial role in the proof of Theorem~\ref{th:main} and Theorem~\ref{th: stronger bound for imag part}. To formulate the main result we shall introduce additional notations. We denote
\begin{equation}\label{eq: Psi definition}
\Psi(z): = \imag s(z) + \frac{p}{nv},
\end{equation}
and recall definition of the region $\mathbb D$,
\begin{equation*}
\mathbb D: = \{z = u+iv \in \C: |u| \le u_0, V \geq v \geq v_0: = A_0 n^{-1} \log n  \},
\end{equation*}
where $u_0, V > 0 $ are any fixed real numbers and $A_0$ is some large constant determined below. The main result of this section is the following lemma.
\begin{lemma}\label{main lemma}
Assuming the conditions $\CondTwo$ there exist  positive constants $C_0, H_0$ depending on $u_0, V$ and positive constants $A_0, A_1$ depending on $C_0, H_0, \delta$ and $K$ (see the next section for definition of $K$)
such that for all $z \in \mathbb D$ and $1 \le p \le A_1 \log n$ we have
\begin{align}\label{eq: main lemma first statement }
\max_{j \in \T} \E|\RR_{jj}(z)|^p &\le C_0^p, \\
\label{eq: main lemma second statement}
\E \frac{1}{|z+m_n(z)|^p} &\le C_0^p,\\
\label{eq: main lemma third statement}
\max_{j \in \T }\E \imag^p \RR_{jj}(z) &\le H_0^p \Psi^p(z).
\end{align}
\end{lemma}
Note that the region $\mathbb D$ and $p$ depend on $\alpha$ via some constants $A_0$ and $A_1$ only, but the power of the logarithmic factor in the definition of  $\mathbb D$ is independent of $\alpha$.
We split the proof of this lemma into two subsections. In the first subsection we prove~\eqref{eq: main lemma first statement } and~\eqref{eq: main lemma second statement}. In the second subsection we prove~\eqref{eq: main lemma third statement}. The proofs are very similar and we shall give the proof of~\eqref{eq: main lemma first statement } in full detail while only sketching the proofs of~\eqref{eq: main lemma second statement} and~\eqref{eq: main lemma third statement}. Note that in the subsequent sections we keep the notations for the  constants $A_0, A_1, C_0$ and $H_0$ introduced above, but the particular values may  depend on the location.

\subsection{Moment bounds for the absolute value of resolvent entries}
We start this section with the following lemma which was proved in~\cite{GotzeNauTikh2015a}[Lemma~4.1].
\begin{lemma}\label{main lemma non optimal}
Assuming the conditions $\CondTwo$ there exist a positive constant $C_0$ depending on $u_0, V$ and positive constants $A_0, A_1$ depending on $C_0, \delta$
such that for all $z \in \mathbb D$ and $1 \le p \le A_1(nv)^\kappa$ we have
\begin{equation}\label{eq: main lemma non optimal first statement}
\max_{j,k \in \T} \E|\RR_{jk}(z)|^p \le C_0^p
\end{equation}
and
\begin{equation}\label{eq: main lemma non optimal second statement}
\E \frac{1}{|z+m_n(z)|^p} \le C_0^p.
\end{equation}
\end{lemma}
\begin{proof}
See~\cite{GotzeNauTikh2015a}[Lemma~4.1].
\end{proof}

\begin{proof}[Proof of~\eqref{eq: main lemma first statement } and~\eqref{eq: main lemma second statement}] Since $u$ is fixed and $|u| \le u_0$ we shall omit $u$ from the notation of the resolvent and denote $\RR(v): = \RR(z)$. Sometimes in order to simplify notations we shall also omit the  argument $v$ in $\RR(v)$ and just write $\RR$. The same applies to the Stieltjes transform $m_n(z)$.

Let $\tilde v = n^{-1} \log^\frac{1}{\kappa} n$, then $1 \le p \le A_1 (n \tilde v)^\kappa = A_1 \log n$. Without loss of generality we may consider $p = A_1 \log n$ only (otherwise one may apply Lyapunov's inequality for moments). It follows from~\eqref{eq: main lemma non optimal first statement} that
$$
\max_{j \in \T} \E |\RR_{jj}(v)|^p \le C_0^p
$$
for all $V \geq v \geq \tilde v$. To prove~\eqref{eq: main lemma first statement } we need to descent from $\tilde v$ to $v_0$ while keeping $p = A_1 \log n$.   Applying Lemma~\ref{lemma: descent} below with $s_0: = \log^{\frac{1}{\kappa}-1} n$ we may show that for all $v \geq v_0$
$$
\max_{j \in \T} \E |\RR_{jj}(v)|^p \le C_0^p \log^{\left(\frac{1}{\kappa}-1\right) p} n.
$$
It remains to remove the log factor on the right hand side of the previous inequality. To this aim we shall adopt the moment matching technique which has been successfully  used  recently by Lee and in Yin in~\cite{LeeYin2014}(see Lemma~5.2 and Lemma~5.3). We denote by $Y_{jk}, 1 \le j \le k \le n$ a triangular set of random variables such that $|Y_{jk}| \le D$, for some $D$ chosen later, and
$$
\E X_{jk}^s = \E Y_{jk}^s \, \text { for } \, s = 1, ... , 4.
$$
It follows from Lemma~\ref{lem: random variables} below that such a set of random variables exists. Let us denote $\W^{\y}: = \frac{1}{\sqrt n} \Y, \RR^\y: = (\W^\y - z \II)^{-1}$ and $m_n^\y(z): = \frac{1}{n} \Tr\RR^\y(z)$. Then, in Lemma~\ref{lem: bound in the optimal region} we show that for all $v \geq v_0$  and $5 \le p \le A_1 \log n$ there exist positive constants $C_1, C_2$ such that
$$
\E|\RR_{jj}(v)|^p \le C_1^p + C_2 \E|\RR_{jj}^\y(v)|^p.
$$
It is easy to see that $Y_{jk}$ are sub-Gaussian random variables. It follows from Lemma~\ref{main lemma gauss}  of the following section that
$$
\E|\RR_{jj}^\y(v)|^p \le C_0^p.
$$
This fact concludes the proof of~\eqref{eq: main lemma first statement }. Similar arguments yield the estimate~\eqref{eq: main lemma second statement}.
\end{proof}
\begin{lemma}\label{lemma: descent}
Let $C_0$ be some constant. Assume that for all $v' \geq \tilde v$ and all $1 \le p \le A_1(nv')^\kappa$ we have
\begin{equation}\label{eq: main assumption}
\max_{j,k \in \T} \E |\RR_{jk}(v')|^p \le C_0^p \quad \text{ and } \quad \E \frac{1}{|z+m_n(v')|^p} \le C_0^p.
\end{equation}
Then for all $s_0 \geq 1, v \geq \tilde v/s_0$ we have
$$
\max_{j, k \in \T} \E |\RR_{jk}(v)|^p \le 2^p(1+s_0)^p C_0^p \quad \text{ and } \quad \E \frac{1}{|z+m_n(v)|^p} \le s_0^p C_0^p.
$$
\end{lemma}
\begin{proof}
We first consider the diagonal entries. Let us fix $s_0\geq 1$ and $v \geq \tilde v/s_0$ and denote $v' = s_0 v \geq \tilde v$. Applying Lemma~\ref{appendix lemma resolvent relations on different v} and the main assumption~\eqref{eq: main assumption}  we get
\begin{equation}\label{eq: diagonal RR}
|\RR_{jj}(v)|^p \le s_0^p |\RR_{jj}(s_0 v)| \le (C_0 s_0)^p.
\end{equation}
Applying the resolvent equality
$$
|\RR_{jk}(v) - \RR_{jk}(s_0v)| \le v(s_0-1) |[\RR(v) \RR(s_0v)]_{jk}|.
$$
The Cauchy-Schwartz inequality and Lemma~\ref{appendix lemma resolvent inequalities 1} together imply that
$$
|\RR_{jk}(v) - \RR_{jk}(s_0v)| \le \sqrt{s_0} \sqrt{|\RR_{jj}(v)| |\RR_{kk}(s_0 v)|}.
$$
It follows that
$$
\E |\RR_{jk}(v)|^p \le 2^p \E |\RR_{jk}(s_0 v)|^p + 2^p  s_0^\frac{p}{2} \E^\frac12 |\RR_{jj}(v)|^p \E^\frac12 |\RR_{kk}(s_0 v)|^p.
$$
Applying~\eqref{eq: diagonal RR} we get
$$
\E |\RR_{jk}(v)|^p \le 2^p C_0^p + 2^p s_0^p C_0^p \le 2^p(1+s_0)^p C_0^p.
$$
Similarly, applying Lemma~\ref{appendix lemma resolvent relations on different v}[Inequality~\eqref{appendix eq resolvent relations on different v 2}], we arrive at a bound for $\E \frac{1}{|z+m_n(v)|^p}$. We omit the details.
\end{proof}

\begin{lemma}\label{lem: random variables}
For any $A, B \in \R$ with $B \geq A^2 + 1$, there exists a random variable $X$ such that
$$
\E X = 0, \, \E X^2 = 1,\,  \E X^3 = A,\,  \E X^4 = B
$$
and
$$
\supp(X) \in [-D B, D B],
$$
for some $D$ independent of $B$.
\end{lemma}
\begin{proof}
See~\cite{LeeYin2014}[Lemma~5.2].
\end{proof}

\begin{lemma}\label{lem: bound in the optimal region}
For all $v \geq v_0$  and $5 \le p \le \log n$ there exist positive constants $C_1, C_2$ such that
\begin{equation}\label{eq: matching 1}
\E|\RR_{jk}(v)|^p \le C_1^p + C_2\E|\widetilde \RR_{jk}(v)|^p
\end{equation}
and
\begin{equation}\label{eq: matching 2}
\E \frac{1}{|z+m_n(v)|^p} \le C_1^p + C_2\E \frac{1}{|z+\tilde m_n(v)|^p}.
\end{equation}
\end{lemma}
\begin{proof}
We first prove~\eqref{eq: matching 1}. The method is based on the following replacement scheme, which has been used in  recent results~\cite{ErdKnowYauYin2012} and~\cite{LeeYin2014}. We replace  the $\frac{n(n+1)}{2}$ matrix entries $X_{ab}$ by $Y_{ab}, 1 \le a \le b \le n$, thus  replacing the corresponding  resolvent entries $\RR_{jk}$ by $ \RR_{jk}^\y$ for every pair of $(j, k)$. Let $\J, \K \subset \T$.
Denote by $\W^{(\J, \K)}$ the  random matrix $\W$ with all entries in the positions $(\mu, \nu), \mu \in \J, \nu \in \K$ replaced by $\frac{1}{\sqrt n} Y_{\mu \nu}$. Assume that we have already exchanged all entries in positions $(\mu, \nu), \mu \in \J, \nu \in \K$ and are going to replace an additional entry in the position $(a, b), a \in \T \setminus \J, b \in \T \setminus \K$. Without loss of generality we may assume that $\J = \emptyset, \K = \emptyset$ (hence $\W^{(\J, \K)} = \W$) and then denote $\V: = \W^{(\{a\}, \{b\})}$. The following  additional notations will be needed.
$$
\EE^{(a,b)} =
\begin{cases}
\ee_{a} \ee_b^\mathsf{T} + \ee_b \ee_a^\mathsf{T}, & 1 \le a < b \le n, \\
\ee_{a} \ee_a^\mathsf{T}, &a = b.
\end{cases}
$$
and $\U : = \W - \frac{X_{ab}}{\sqrt n}  \EE^{(a,b)}$, where $\ee_j$ denotes a unit column-vector with all zeros except $j$-th position. In these notations we may write
$$
\W = \U + \frac{1}{\sqrt n} X_{ab} \EE^{(a,b)}, \quad \V = \U + \frac{1}{\sqrt n} Y_{ab} \EE^{(a,b)}.
$$
Recall that $\RR: = (\W - z \II)^{-1}$ and denote $\RS: = (\V - z \II)^{-1}$ and $\RT: = (\U - z \II)^{-1}$. Let us assume that we have already proved the following fact
\begin{equation}\label{eq: main eq 1}
\E|\RR_{jk}|^p = \mathcal I(p) + \frac{\theta_1 C^p}{n^2} + \frac{\theta_1\E|\RR_{jk}|^p}{n^2},
\end{equation}
where $\mathcal I(p)$ is some quantity depending on $p, n$ (see~\eqref{eq: mathcal A(p) def} below for precise definition)  and $|\theta_1|\le 1, C > 0$ are some numbers. Similarly,
\begin{equation}\label{eq: main eq 2}
\E|\RS_{jk}|^p = \mathcal I(p) + \frac{\theta_2 C^p}{n^2} + \frac{\theta_2\E|\RS_{jk}|^p}{n^2},
\end{equation}
where $|\theta_2| \le 1$. It follows from~\eqref{eq: main eq 1} and~\eqref{eq: main eq 2} that
\begin{equation}
\left(1 - \frac{\theta_1}{n^2}\right )\E|\RR_{jk}|^p \le \left(1 - \frac{\theta_2}{n^2}\right )\E|\RS_{jk}|^p + \frac{2C^p}{n^2},
\end{equation}
Let us denote $\rho : = \left(1 - \frac{\theta_2}{n^2} \right )\left(1 - \frac{\theta_1}{n^2}\right )^{-1}$. We get
\begin{equation}\label{eq: step}
\E|\RR_{jk}|^p \le \rho \E|\RS_{jk}|^p + \frac{C_1^p}{n^2},
\end{equation}
with some positive constant $C_1$. Repeating~\eqref{eq: step} recursively $n(n+1)/2$ times we arrive at the following bound
\begin{equation}\label{eq: step 2}
\E|\RR_{jk}|^p \le \rho^\frac{n(n+1)}{2} \E|\RR_{jk}^\y|^p + \frac{C_1^p}{n^2} \left(1 + \rho_1 + ... + \rho_1^{\frac{n(n+1)}{2} - 1 }\right).
\end{equation}
It is easy to see from the definition of $\rho$ that for some $\theta$, say $|\theta| < 4$, we have
$$
\rho \le 1 + \frac{|\theta|}{n^2}.
$$
From this inequality and~\eqref{eq: step 2} we deduce that
\begin{equation}
\E|\RR_{jk}|^p \le C_2 \E|\RR_{jk}^\y|^p + C_3^p,
\end{equation}
with some positive constants $C_2$ and $C_3$. From the last inequality we may conclude the statement of the lemma. It remains to prove~\eqref{eq: main eq 1} (resp.~\eqref{eq: main eq 2}).  Applying the resolvent equation we get for $m \geq 0$
\begin{equation}\label{resolvent expansion R}
\RR = \RT + \sum_{\mu=1}^m \frac{(-1)^\mu}{n^\frac{\mu}{2}} X_{ab}^\mu (\RT \EE^{(a,b)})^\mu \RT + \frac{(-1)^{m+1}}{n^\frac{m+1}2} X_{ab}^{m+1} (\RT \EE^{(a,b)})^{m+1} \RR.
\end{equation}
The same identity holds for $\RS$
\begin{equation}\label{resolvent expansion S}
\RS = \RT + \sum_{\mu=1}^m \frac{(-1)^\mu}{n^\frac{\mu}{2}} Y_{ab}^\mu (\RT \EE^{(a,b)})^\mu \RT + \frac{(-1)^{m+1}}{n^\frac{m+1}{2}} Y_{ab}^{m+1} (\RT \EE^{(a,b)})^{m+1} \RS.
\end{equation}
We investigate~\eqref{resolvent expansion R}. In order handle arbitrary high moments of $\RR_{jk}$ we apply a Stein type technique which we have  used  in previous papers~\cite{GotzeNauTikh2015a} and ~\cite{GotzeTikh2014rateofconv}. Let us introduce the following function $\varphi(z): = \overline z |z|^{p-2}$ and write
$$
\E |\RR_{jk}|^p = \E \RR_{jk} \varphi(\RR_{jk}).
$$
Applying~\eqref{resolvent expansion R}  we get
\begin{align}\label{eq: p moment of R_{jk}}
\E |\RR_{jk}|^p &= \sum_{\mu=0}^4 \frac{(-1)^\mu}{n^\frac{\mu}{2}} \E X_{ab}^\mu [(\RT \EE^{(a,b)})^\mu \RT]_{jk} \varphi(\RR_{jk})
\nonumber\\
&+\sum_{\mu=5}^m \frac{(-1)^\mu}{n^\frac{\mu}{2}} \E X_{ab}^\mu [(\RT \EE^{(a,b)})^\mu \RT]_{jk} \varphi(\RR_{jk}) \nonumber \\
&+  \frac{1}{n^\frac{m+1}2} \E X_{ab}^{m+1} [(\RT \EE^{(a,b)})^{m+1} \RR]_{jk}\varphi(\RR_{jk}) =: \mathcal A_0 + \mathcal A_1 + \mathcal A_2.
\end{align}
The bound for $\mathcal A_2$ is easy. It is straightforward to check that $[(\RT \EE^{(a,b)})^{m+1} \RR]_{jk}$ is the sum of $2^{m+1}$ terms of the following type
$$
\RT_{j i_1} \RT_{i_1 i_2} ... \RT_{i_m i_m} \RR_{i_m k},
$$
where $i_l = a$ or $i_l = b$ for $l = 1, ... , m$. We may estimate $|\RR_{ak}|$ or $|\RR_{bk}|$ by $n$, since $v > n^{-1}$. Taking absolute values, applying H{\"o}lder's inequality and the bound $\E^\frac{1}{c p} |\RT_{i_l i_{l+1}}|^{c p} < C_0 \log^{\frac{1}{\kappa}-1}n$ for some $c>0$ (see Lemma~\ref{lemma: descent} with $s_0: = \log^{\frac{1}{\kappa}-1} n$) we get
$$
|\mathcal A_2| \le \frac{1}{n^\frac{m+1}2} \E|X_{ab}|^{m+1} [(\RT \EE^{(a,b)})^{m+1} \RR]_{jk}|\RR_{jk}|^{p-1} \le  \frac{n \log^{c(\kappa)}n}{n^\frac{m+1}2} \E^\frac{1}{p}|X_{ab}|^{(m+1)p} \E^\frac{p-1}{p}|\RR_{jk}|^{p},
$$
where $c(\kappa) > 0 $ -- some constant depending on $\kappa$. Using $\CondTwo$ we may assume without loss of generality that $|X_{jk}| \le D n^{\frac12 - \phi}$ for some $\phi > 0$ depending on $\delta$. Indeed, one may choose $\phi: = \frac{\delta}{2(4+\delta)}$. Applying this fact and the last inequality we may write
$$
|\mathcal A_2| \le \frac{C n \log^{c(\kappa)}n}{n^{\phi(m+1)}}  \E^\frac{p-1}{p}|\RR_{jk}|^{p}.
$$
We now choose $m$ such that $\phi(m+1) = 4$. Applying $ 2 = \frac{2}{p} + \frac{2(p-1)}{p}$ and Young's inequality we obtain
\begin{equation}\label{eq: A_2 estimate}
|\mathcal A_2| \le \frac{C}{n^{\frac{2}{p}+\frac{2(p-1)}{p}}} \E^\frac{p-1}{p}|\RR_{jk}|^{p} \le \frac{C^p}{p n^{2}} + \frac{\E|\RR_{jk}|^{p}}{n^2} .
\end{equation}
Let us consider the term $\mathcal A_1 = \mathcal A_{1,5} + ... + \mathcal A_{1, m}$, where
$$
A_{1, \mu} : = \frac{(-1)^\mu}{n^\frac{\mu}{2}} \E X_{ab}^\mu [(\RT \EE^{(a,b)})^\mu \RT]_{jk} \varphi(\RR_{jk}), \quad \mu = 5, \ldots, m.
$$
We fix some $\mu \geq 5$ and bound  $\mathcal A_{1,\mu}$.
Let us introduce the following quantity
\begin{equation}\label{tilde T definition}
\widetilde \RT_{jk}: = \E(\RR_{jk}\big|\mathfrak M^{(a,b)}),
\end{equation}
where $\mathfrak M^{(a,b)}: = \sigma\{X_{st}, (s,t) \neq (a,b)\}$ -- $\sigma$-algebra generated by  $X_{st}, 1 \le s \le t \le n, (s,t) \neq (a,b)$. We may split $\mathcal A_{1,\mu}$ into a sum of two terms $\mathcal A_{1,\mu} = \mathcal B_1 + \mathcal B_2$, where
\begin{align*}
&\mathcal B_{1} = \frac{(-1)^\mu}{n^\frac{\mu}{2}} \E X_{ab}^\mu [(\RT \EE^{(a,b)})^\mu \RT]_{jk} \varphi(\widetilde \RT_{jk}), \\
&\mathcal B_{2} = \frac{(-1)^\mu}{n^\frac{\mu}{2}} \E X_{ab}^\mu [(\RT \EE^{(a,b)})^\mu \RT]_{jk} [\varphi(\RR_{jk}) -  \varphi(\widetilde \RT_{jk})].
\end{align*}
For the first term we may write
\begin{align*}
|\mathcal B_{1}| &\le \frac{1}{n^\frac{\mu}{2}} \E |X_{ab}|^\mu \E|[(\RT \EE^{(a,b)})^\mu \RT]_{jk}| |\widetilde \RT_{jk}|^{p-1} \le \frac{C \log^{c(\kappa)} n}{n^\frac{\mu}{2}} \E |X_{ab}|^\mu \E^\frac{p-1}{p} |\RR_{jk}|^{p}, \\
&\le \frac{C \log^{c(\kappa)} n}{n^{2 + \phi(\mu-4)}} \E^\frac{p-1}{p} |\RR_{jk}|^{p} \le \frac{C}{n^{\frac{2}{p} + \frac{2(p-1)}{p} }} \E^\frac{p-1}{p} |\RR_{jk}|^{p} \le \frac{C^p}{n^2} + \frac{\E|\RR_{jk}|^{p}}{n^2}.
\end{align*}
Let us introduce the following function
\begin{equation}\label{eq: function f}
f(t): = \varphi(\widetilde \RT_{jk} + t(\RR_{jk}-\widetilde \RT_{jk})), \quad t \in [0,1].
\end{equation}
It is easy to see that
$$
f(1) = \varphi(\RR_{jk}), \quad f(0) = \varphi(\widetilde \RT_{jk}). 
$$
Moreover,  by Taylor's formula, for all $l \geq 1$,
\begin{equation}\label{eq: Taylors formula}
f(t) = \sum_{k=0}^{l-1} \frac{f^{(k)}(0)}{k!} t^k + \frac{t^l}{(l-1)!} \E_\theta (1-\theta)^{l-1} f^{(l)}(\theta t),
\end{equation} 
where $\theta$ is uniformly distributed on $[0, 1]$ and independent of all other r.v. 
It is straightforward to check that there exists some positive constant $K_l$ depending on $l$ such that
$$
|f^{(l)}(t)| \le K_l p^{l}|t|^l |\widetilde \RT_{jk} + t (\RR_{jk}-\widetilde \RT_{jk})|^{p-l-1} .
$$
Applying~\eqref{eq: Taylors formula} with $t = 1, l = 1$ we estimate $\mathcal B_2$ as follows
$$
|\mathcal B_{2}| \le \frac{p}{n^\frac{\mu}{2}} \E |X_{ab}|^\mu |[(\RT \EE^{(a,b)})^\mu \RT]_{jk}| |\RR_{jk} -  \widetilde \RT_{jk}||\widetilde \RT_{jk} + \theta(\RR_{jk} - \widetilde \RT_{jk})|^{p-2}.
$$
Using triangular inequalities we may bound (up to some constant) the r.h.s. of the previous inequality by the sum $|\mathcal B_2| \le \mathcal I_1 + \mathcal I_2$, where
\begin{align*}
&\mathcal I_{1} = \frac{p}{n^\frac{\mu}{2}} \E |X_{ab}|^\mu |[(\RT \EE^{(a,b)})^\mu \RT]_{jk}| |\RR_{jk} -  \widetilde \RT_{jk}||\widetilde \RT_{jk}|^{p-2}, \\
&\mathcal I_{2} = \frac{p^p}{n^\frac{\mu}{2}} \E |X_{ab}|^\mu |[(\RT \EE^{(a,b)})^\mu \RT]_{jk}| |\RR_{jk} -  \widetilde \RT_{jk}|^{p-1}.
\end{align*}
Applying again~\eqref{resolvent expansion R} we obtain $\mathcal I_1 \le \mathcal I_{11} + ... + \mathcal I_{14}$, where we denoted
\begin{align*}
&\mathcal I_{11} = p \sum_{l=1}^m \frac{1}{n^\frac{\mu+l}{2}} \E |X_{ab}|^{\mu+l} |[(\RT \EE^{(a,b)})^\mu \RT]_{jk}| |[(\RT \EE^{(a,b)})^l \RT]_{jk}|\widetilde \RT_{jk}|^{p-2}, \\
&\mathcal I_{12} = p \sum_{l=1}^m \frac{1}{n^\frac{\mu+l}{2}} \E|X_{ab}|^l \E |X_{ab}|^{\mu} \E|[(\RT \EE^{(a,b)})^\mu \RT]_{jk}| |[(\RT \EE^{(a,b)})^l \RT]_{jk}|\widetilde \RT_{jk}|^{p-2},\\
&\mathcal I_{13} =  \frac{p}{n^\frac{\mu+m+1}{2}} \E |X_{ab}|^{\mu+m+1} |[(\RT \EE^{(a,b)})^\mu \RT]_{jk}| |[(\RT \EE^{(a,b)})^{m+1} \RR]_{jk}||\widetilde \RT_{jk}|^{p-2},\\
&\mathcal I_{14} =  \frac{p}{n^\frac{\mu+m+1}{2}} \E |X_{ab}|^{\mu} |[(\RT \EE^{(a,b)})^\mu \RT]_{jk}| |\E(|X_{ab}|^{m+1}|[(\RT \EE^{(a,b)})^{m+1} \RR]_{jk}|\big|\mathfrak M^{(a,b)})|\widetilde \RT_{jk}|^{p-2}.
\end{align*}
It is straightforward to check that
\begin{align*}
\mathcal I_{11} \le \frac{C m}{n^{\frac{4}{p} + \frac{2(p-2)}{p} }}\E^\frac{p-2}{p}|\RR_{jk}|^{p} \le \frac{C^p}{n^2} + \frac{\E |\RR_{jk}|^{p}}{n^2}.
\end{align*}
The term $\mathcal I_{12}$ may be estimated in a similar way.  The terms $\mathcal I_{13}$ and $\mathcal I_{14}$ can be estimated as before choosing $m$ sufficiently large.
Let us consider the term $\mathcal I_2$. Similarly, it may be rewritten as follows
\begin{align*}
&\mathcal I_{21} =  \sum_{l=1}^m \frac{p^p}{n^\frac{\mu+l(p-1)}{2}} \E |X_{ab}|^{\mu+l(p-1)} \E|[(\RT \EE^{(a,b)})^\mu \RT]_{jk}| |[(\RT \EE^{(a,b)})^l\RT]_{jk}|^{p-1}, \\
&\mathcal I_{22} =  \sum_{l=1}^m \frac{p^p}{n^\frac{\mu+l(p-1)}{2}} \E^{p-1} |X_{ab}|^l \E |X_{ab}|^{\mu} \E|[(\RT \EE^{(a,b)})^\mu \RT]_{jk}| |[(\RT \EE^{(a,b)})^l\RT]_{jk}|^{p-1}, \\
&\mathcal I_{23} =  \frac{p^p}{n^\frac{\mu+(m+l)(p-1)}{2}} \E |X_{ab}|^{\mu+(m+1)(p-1)} |[(\RT \EE^{(a,b)})^\mu \RT]_{jk}| |[(\RT \EE^{(a,b)})^{m+1} \RR]_{jk}|^{p-1},\\
&\mathcal I_{24} =  \frac{p^p}{n^\frac{\mu+(m+l)(p-1)}{2}} \E |X_{ab}|^{\mu} |[(\RT \EE^{(a,b)})^\mu \RT]_{jk}| |\E(|X_{ab}|^{m+1}|[(\RT \EE^{(a,b)})^{m+1} \RR]_{jk}|\big|\mathfrak M^{(a,b)})|^{p-1}.
\end{align*}
We investigate $\mathcal I_{21}$. We obtain
$$
\mathcal I_{21} \le   \sum_{l=1}^m \frac{C^{cp} p^p \log^{c(\kappa)p} n }{n^{2+ \phi l(p-1)}} \le \frac{C^p}{n^2}.
$$
The same estimate is valid for $\mathcal I_{22}$. The terms $\mathcal I_{23}, \mathcal I_{24}$ may be estimated as before choosing  $m$ sufficiently large. Finally, we  conclude that
\begin{equation}\label{eq: A_1 estimate}
|\mathcal A_1| \le \frac{C^p}{n^2} + \frac{\E|\RR_{jk}|^{p}}{n^2}.
\end{equation}
Let us consider now the term $\mathcal A_0$. Applying~\eqref{eq: Taylors formula} with $t = 1, l = 5-\mu$ we rewrite $\mathcal A_0$ as follows
\begin{align*}
\mathcal A_0 &= \sum_{\mu=0}^4 \frac{(-1)^\mu}{n^\frac{\mu}{2}} \E X_{ab}^\mu \E [(\RT \EE^{(a,b)})^\mu \RT]_{jk} \varphi(\RT_{jk}) \\
&+ \sum_{\mu=0}^4 \frac{(-1)^\mu}{n^\frac{\mu}{2}} \sum_{l=1}^{4-\mu} \frac{1}{l!}\E X_{ab}^\mu  [(\RT \EE^{(a,b)})^\mu \RT]_{jk} (\RR_{jk} - \RT_{jk})^l  \varphi^{(l)}(\RT_{jk}) \\
&+  \sum_{\mu=0}^4 \frac{(-1)^\mu}{n^\frac{\mu}{2}} \frac{1}{(4-\mu)!} \E X_{ab}^\mu [(\RT \EE^{(a,b)})^\mu \RT]_{jk} (\RR_{jk} - \RT_{jk})^{5-k} (1 - \theta)^{4-k} \\
&\qquad\qquad\qquad\qquad\qquad\qquad\times\varphi^{(5-\mu)}(\RT_{jk} + \theta(\RR_{jk} - \RT_{jk}))\\
& =: \mathcal A_{00} + \mathcal A_{01} + \mathcal A_{02}.
\end{align*}
Let us study the term $\mathcal A_{01}$. We may write $\mathcal A_{01} = \sum_{\mu=0}^4 \sum_{l=1}^{4-\mu} (-1)^\mu (l!)^{-1} \mathcal B_{\mu l}$. We fix some $\mu = 0 ,... , 4$ and some $l = 1,..., 4-\mu$. We may apply now~\eqref{resolvent expansion R} with $m \geq 4$ and get $\mathcal B_{\mu l} = \mathcal B_{\mu l}^{(0)} + \mathcal B_{\mu l}^{(1)}+\mathcal B_{\mu l}^{(2)}$, where
\begin{align*}
&\mathcal B_{\mu l}^{(0)} := \sum_{\substack{\mu_1 + ... + \mu_m = l \\ \mu+\mu_1 + 2\mu_2 + ... + m \mu_m \le 4 }} \frac{C_{\mu_1, ... , \mu_m}^l}{n^{\frac{\mu}{2} + \frac{\mu_1}{2} + \frac{2\mu_2}{2} + ... + \frac{m \mu_m}{2}}} \E X_{ab}^{\mu+\mu_1 + 2\mu_2 + ... + m \mu_m} \\
&\qquad\qquad\qquad\qquad\times \E [(\RT \EE^{(a,b)})^\mu \RT]_{jk}[(\RT \EE^{(a,b)})\RT]_{jk}^{\mu_1}...[(\RT \EE^{(a,b)})^m\RT]_{jk}^{\mu_m} \varphi^{(l)}(\RT_{jk}), \\
&\mathcal B_{\mu l}^{(1)} := \sum_{\substack{\mu_1 + ... + \mu_m = l \\ \mu+ \mu_1 + 2\mu_2 + ... + m \mu_m > 4 }} \frac{C_{\mu_1, ... , \mu_m}^l}{n^{\frac{\mu}{2}+\frac{\mu_1}{2} + \frac{2\mu_2}{2} + ... + \frac{m \mu_m}{2}}} \E X_{ab}^{\mu+\mu_1 + 2\mu_2 + ... + m \mu_m} \\
&\qquad\qquad\qquad\qquad\times \E [(\RT \EE^{(a,b)})^\mu \RT]_{jk}[(\RT \EE^{(a,b)})\RT]_{jk}^{\mu_1}...[(\RT \EE^{(a,b)})^m \RT]_{jk}^{\mu_m} \varphi^{(l)}(\RT_{jk}), \\
&\mathcal B_{\mu l}^{(2)} := \sum_{\mu_{m+1} = 1}^l \sum_{\mu_1 + ... + \mu_m = l-\mu_{m+1}} \frac{C_{\mu_1, ... , \mu_{m+1}}^l}{n^{\frac{\mu}{2} +\frac{\mu_1}{2} + \frac{2\mu_2}{2} + ... + \frac{(m+1) \mu_{m+1}}{2}}} \E X_{ab}^{\mu+\mu_1 + 2\mu_2 + ... + (m+1) \mu_{m+1}} \\
&\qquad\qquad\times  [(\RT \EE^{(a,b)})^\mu \RT]_{jk}[(\RT \EE^{(a,b)})\RT]_{jk}^{\mu_1}...[(\RT \EE^{(a,b)})^m\RT]_{jk}^{\mu_m}[(\RT \EE^{(a,b)})^{m+1}\RR]_{jk}^{\mu_{m+1}} \varphi^{(l)}(\RT_{jk}),
\end{align*}
where $C_{\mu_1, ... , \mu_{m+1}}^l$ are multinomial coefficients. We now introduce in~\eqref{eq: main eq 1}
\begin{equation}\label{eq: mathcal A(p) def}
\mathcal I(p) := \mathcal A_{00} + \sum_{\mu=0}^4 \sum_{l=1}^{4-\mu} \frac{(-1)^\mu}{l!} \mathcal B_{\mu l}^{(0)}.
\end{equation}
The estimate of $\mathcal B_{\mu l}^{(1)}$ is similar to $\mathcal A_1$, see~\eqref{eq: A_1 estimate}. Here, one has to expand   $\varphi^{(l)}(\RT_{jk})$ at the point $\widetilde \RT_{jk}$ and apply the same arguments as before. To estimate $\mathcal B_{\mu l}^{(2)}$ we again expand $\varphi^{(l)}(\RT_{jk})$ at the point $\widetilde \RT_{jk}$. From here on  we may apply the same arguments as for $\mathcal A_1$ and $\mathcal A_2$ (see~\eqref{eq: A_1 estimate} and~\eqref{eq: A_2 estimate} by taking $m$ large enough and bound $|\RR_{ak}|$ or $|\RR_{bk}|$ by $n$. The same procedure applies to $\mathcal A_{02}$. We finally get
$$
\E|\RR_{jk}|^p = \mathcal I(p) + r_n(p),
$$
where
$$
|r_n(p)| \le \frac{C^p}{n^2} + \frac{\E|\RR_{jk}|^{p}}{n^2}.
$$
The proof of~\eqref{eq: matching 2} is similar. Let us denote by $m_n^{(a,b)}(z)$ the Stieltjes transform of $\W^{(\{a\}, \{b\})}$. It is easy to see that
$$
\frac{1}{z + m_n(z)} = \frac{1}{z + m_n^{(a,b)}(z)} + \frac{m_n(z) - m_n^{(a,b)}(z)}{(z + m_n(z))(z + m_n^{(a,b)}(z))}.
$$
We may use this formula recursively together with~\eqref{resolvent expansion R} and get~\eqref{eq: matching 2}. We omit the details.
\end{proof}

\subsection{Moments of  the imaginary part of diagonal entries of the resolvent}

\begin{lemma}\label{lemma: imag part of R_jj}
Assuming conditions $\CondTwo$ there exist positive constants $H_0$ depending on $u_0, V$ and positive constants $A_0, A_1$ depending on $H_0, \delta$ such that for all $1 \le p \le A_1(nv)^\kappa$ and $z \in \mathbb D$  we get
$$
\max_{j \in \T }\E \imag^p \RR_{jj}(z) \le H_0^p \imag^p s(z) + \frac{H_0^p p^{2p}}{(nv)^p}.
$$
\end{lemma}
\begin{proof}
See~\cite{GotzeNauTikh2015a}[Lemma~5.1].
\end{proof}
\begin{proof}[Proof of~\eqref{eq: main lemma third statement}]
Let $\tilde v = n^{-1} \log^\frac{1}{\kappa} n$, then $1 \le p \le A_1 (n \tilde v) = A_1 \log n$. Without loss of generality we may consider $p = A_1 \log n$ only . Using the notation introduced in \eqref{eq: Psi definition} the statement of  Lemma~\ref{lemma: imag part of R_jj} may be rewritten (up to constant) in the following way
\begin{equation}\label{eq: crude bound}
\max_{j \in \T }\E \imag^p \RR_{jj}(z) \le H_0^p p^p \Psi^p(z) \le H_0^p \log^p n \Psi^p(z)
\end{equation}
for all $V \geq v \geq \tilde v$. To prove~\eqref{eq: main lemma first statement } we need to descent from $\tilde v$ to $v_0$ while maintaining that $p = A_1 \log n$.   Applying Lemma~\ref{lemma: descent imag} below with $s_0: = \log^{\frac{1}{\kappa}-1} n$ we may show that for all $v \geq v_0$
$$
\max_{j \in \T} \E \imag^p \RR_{jj}(v) \le H_0^p \log^{\left(\frac{2}{\kappa}-1\right) p} n \Psi^p(z).
$$
It remains to remove the log factor from the r.h.s. on the previous equation. We may proceed as before and arrive at the following bound for any $j \in \T$
$$
\E \imag^p \RR_{jj}(v) \le C_1^p \Psi^p(z) + C_2 \E \imag^p \RR_{jj}^\y(v).
$$
See Lemma~\ref{lem: bound in the optimal region imag} below. It follows from Lemma~\ref{lemma: imag part of R_jj gauss} in the next section that
$$
\max_{j \in \T} \E \imag^p \RR_{jj}^\y(v) \le H_0^p \Psi^p(z).
$$
\end{proof}

\begin{lemma}\label{lemma: descent imag}
Let $C_0$ be some constant. Assume that for all $v' \geq \tilde v$ and all $1 \le p \le A_1(nv')^\kappa$ we have
\begin{equation}\label{eq: main assumption imag }
\max_{j \in \T} \E \imag^p \RR_{jj}(v') \le H_0^p \log^p n \Psi^p(z).
\end{equation}
Then for all $s_0 \geq 1, v \geq \tilde v/s_0$ we have
$$
\max_{j, k \in \T} \E \imag^p \RR_{j k}(v) \le s_0^{2p} H_0^p \log^p n \Psi^p(z).
$$
\end{lemma}
\begin{proof}
The proof is similar to the proof of Lemma~\ref{lemma: descent}. Applying Lemma~\ref{appendix lemma imag resolvent relations on different v} we get
$$
\E \imag^p \RR_{j j}(v) \le s_0^p \E \imag^p \RR_{j j}(s_0 v) \le s_0^p H_0^p \Psi^q(s_0 v) \le s_0^{2p} H_0^p \log^p n \Psi^q(v).
$$
We consider now the case $j \neq k$. Then
$$
\imag \RR_{jk} = v [\RR \RR^{*}]_{jk}  = v \sum_{l = 1}^n \RR_{jl} \RR_{lk}^{*} \le v \left(\sum_{l=1}^n|\RR_{jl}|^2 \right)^\frac12 \left(\sum_{l=1}^n|\RR_{lk}|^2 \right)^\frac12.
$$
Applying Lemma~\ref{appendix lemma resolvent inequalities 1} we get
$$
\imag \RR_{jk} \le \sqrt{\imag \RR_{jj} \imag \RR_{kk}}.
$$
It follows that
$$
\E \imag^p \RR_{j k}(v) \le s_0^{2p} H_0^p \log^p n \Psi^q(v).
$$
\end{proof}

\begin{lemma}\label{lem: bound in the optimal region imag}
For all $v \geq v_0$  and $5 \le p \le \log n$ there exists a constant $C_0>0$ such that
\begin{equation}\label{eq: matching 1 imag}
\E|\imag \RR_{jj}(v)|^p \le C_0^p \Psi^p(z) + \E|\imag \RR_{jj}^\y(v)|^p.
\end{equation}
\end{lemma}
\begin{proof}
The proof of this lemma is similar to the proof of Lemma~\ref{lem: bound in the optimal region}. We shall omit the  details, but emphasize some important points.  Applying~\eqref{resolvent expansion R} with some $m \geq 0$ we get
\begin{align*}
\imag \RR_{jj} &= \imag \RT_{jj} + \sum_{\mu=1}^m \frac{(-1)^\mu}{n^\frac{\mu}{2}} X_{ab}^\mu \imag [(\RT \EE^{(a,b)})^\mu \RT]_{jj} \\
&\qquad\qquad\qquad\qquad + \frac{(-1)^{m+1}}{n^\frac{m+1}2} \X_{ab}^{m+1} \imag[(\RT \EE^{(a,b)})^{m+1} \RR]_{jj}.
\end{align*}
Introduce the  function $\psi(x): = x^{p-1}$ and write $\E \imag^p \RR_{jj} = \E \imag \RR_{jj} \psi(\imag \RR_{jj})$. Similarly to~\eqref{eq: p moment of R_{jk}} we get
\begin{align}\label{eq: p moment of imag R}
\E \imag^p \RR_{jj} &= \sum_{\mu=0}^4 \frac{(-1)^\mu}{n^\frac{\mu}{2}} \E X_{ab}^\mu \imag[(\RT \EE^{(a,b)})^\mu \RT]_{jj} \psi(\imag \RR_{jj})
\nonumber\\
&+\sum_{\mu=5}^m \frac{(-1)^\mu}{n^\frac{\mu}{2}} \E X_{ab}^\mu \imag[(\RT \EE^{(a,b)})^\mu \RT]_{jj} \psi(\imag \RR_{jj}) \nonumber \\
&+  \frac{1}{n^\frac{m+1}2} \E\X_{ab}^{m+1} \imag [(\RT \EE^{(a,b)})^{m+1} \RR]_{jj}\psi(\imag \RR_{jj}) =: \mathcal A_0 + \mathcal A_1 + \mathcal A_2.
\end{align}
We shall keep the same notations as in the  proof of Lemma~\ref{lem: bound in the optimal region}.  Let us consider the term $\mathcal A_2$. Repeating the same arguments as in the proof of Lemma~\ref{lem: bound in the optimal region} for the corresponding terms and applying $\|\RR\| \le v^{-1}$ we get
\begin{align*}
|\mathcal A_2| &\le \frac{1}{n^\frac{m+1}2} \E|X_{ab}|^{m+1} \imag [(\RT \EE^{(a,b)})^{m+1} \RR]_{jj}\imag^{p-1} \RR_{jj} \\
&\le  \frac{\log^{c} n}{n^\frac{m+1}2 v} \E^\frac{1}{p}|\X_{ab}|^{(m+1)p} \E^\frac{p-1}{p} \imag^p \RR_{jj}.
\end{align*}
Since $|X_{jk}| \le D n^{\frac12 - \phi}$ we obtain that
$$
|\mathcal A_2| \le \frac{C \log^{c(\kappa)} n}{n^{\phi(m+1)} v }  \E^\frac{p-1}{p} \imag^p \RR_{jj}.
$$
We may choose $m$ such that $\phi(m+1) = 4$. Applying $ 2 = \frac{2}{p} + \frac{2(p-1)}{p}$, Young's inequality and $(nv)^{-1} \le \Psi(z)$ we obtain
\begin{equation}\label{eq: A_2 estimate imag}
|\mathcal A_2| \le \frac{C \Psi(z)}{n^{\frac{2}{p}+\frac{2(p-1)}{p}}} \E^\frac{p-1}{p}\imag^p \RR_{jj} \le \frac{C^p \Psi^p(z)}{n^{2}} + \frac{\E \imag^p \RR_{jj}}{n^2} .
\end{equation}
Let us consider $\mathcal A_1$ and split it into the sum $\mathcal A_1 = \mathcal A_{1,5} + ... + \mathcal A_{1,m}$. For an arbitrary $5 \le \mu \le m$ we get $\mathcal A_{1,\mu} = \mathcal B_1 + \mathcal B_2$, where
\begin{align*}
&\mathcal B_{1} = \frac{(-1)^\mu}{n^\frac{\mu}{2}} \E X_{ab}^\mu \imag[(\RT \EE^{(a,b)})^\mu \RT]_{jj} \psi(\imag \widetilde \RT_{jj}), \\
&\mathcal B_{2} = \frac{(-1)^\mu}{n^\frac{\mu}{2}} \E X_{ab}^\mu \imag[(\RT \EE^{(a,b)})^\mu \RT]_{jj} [\psi(\imag\RR_{jj}) -  \psi(\imag \widetilde \RT_{jj})]
\end{align*}
and $\widetilde \RT_{jj}$ are defined in~\eqref{tilde T definition}. The term $[(\RT \EE^{(a,b)})^\mu \RT]_{jj}$
is a sum of $2^{\mu}$ terms of the following type
$$
\RT_{j i_1} \RT_{i_1 i_2} ... \RT_{i_{\mu-1} i_{\mu}} \RT_{i_{\mu} j},
$$
where $i_l = a$ or $i_l = b$ for $l = 1, ... , \mu$. The  imaginary part of such a product may be  bounded from above by a product where  at least one factor  is $|\imag \RT_{i_l i_{l+1}}|$. All other factors may be bounded by their absolute values due to the first statement~\eqref{eq: main lemma first statement } of Lemma~\ref{main lemma}. Applying H\"older's inequality to this product and Lemma~\ref{lemma: descent imag} we get
$$
|\mathcal B_1|  \le \frac{C^p \Psi^p(z)}{n^{2}} + \frac{\E \imag^p \RR_{jj}}{n^2} .
$$
By the same arguments we get similar bounds for $\mathcal B_2$ and $\mathcal A_0$. We omit the details.
\end{proof}

\section{Bounds for moments of diagonal entries of the resolvent in the sub-Gaussian case }\label{sec: diagonal entries}
As  mentioned in the previous section we have to bound the moments of the diagonal entries of the resolvent in the sub-Gaussian case. We denote
\begin{equation*}
\widetilde{\mathbb D}: = \{z = u+iv \in \C: |u| \le u_0, V \geq v \geq \tilde v_0: = A_0 n^{-1} \},
\end{equation*}
where $u_0, V > 0 $ are any fixed real numbers and $A_0$ is some large constant determined below. Comparing with $\mathbb D$ we allow to descent to $A_0 n^{-1}$ along $v$.

We say that the conditions $\CondThree$ are satisfied if $X_{jk}$ satisfies the conditions $\Cond$ and have a sub-Gaussian distribution. It is well-known that the random variables $\xi$ are sub-gaussian if and only if $\E |\xi|^p  = O(p^\frac{p}{2})$
as $p \rightarrow \infty$. We define the sub-Gaussian norm of $\xi$ as
$$
\|\xi\|_{\psi_2}: = \sup_{p \geq 1} p^{-\frac12} \E^\frac{1}{p}|\xi|^p.
$$
In what follows we assume that $K: = \|X_{jk}\|_{\psi_2}$.
\begin{lemma}\label{main lemma gauss}
Assuming the conditions $\CondThree$ there exist a positive constant $C_0$ depending on $u_0, V$ and positive constants $A_0, A_1$ depending on $C_0$ and $K$
such that for all $z \in \widetilde{\mathbb D}$ and $1 \le p \le A_1 nv$ we have
\begin{equation}\label{eq: main lemma first statement gauss}
\max_{j,k \in \T} \E|\RR_{jk}(z)|^p \le C_0^p.
\end{equation}
and
\begin{equation}
\label{eq: main lemma second statement gauss}
\E \frac{1}{|z+m_n(z)|^p} \le C_0^p.
\end{equation}
\end{lemma}

The proof of Lemma~\ref{main lemma gauss}  is based on several auxiliary results and will be given at the end of this section. In this proof  will shall  use  ideas from~\cite{GotzeTikh2014rateofconv} and~\cite{GotTikh2015}. One of main ingredients of the proof is the descent method for $\RR_{jj}$ which is based on Lemma~\ref{lemma: recurence relation for R_jj} below and Lemma~\ref{appendix lemma resolvent relations on different v} in the Appendix, which in this form appeared in~\cite{Schlein2014}. Comparing the result of Lemma~\ref{main lemma gauss} with~\cite{Schlein2014}[Lemma~3.4] that in the latter the power $p$ is bounded from above by $(nv)^\frac14$, which is non-optimal.

Since $u$ is fixed and $|u| \le u_0$ we shall omit $u$ from the notation of the resolvent and denote $\RR(v): = \RR(z)$. Sometimes in order to simplify notations we shall also omit the  argument $v$ in $\RR(v)$ and just write $\RR$.

We also estimate the moments of the imaginary part of the diagonal entries of the  resolvent.  Recall that (see definition~\eqref{eq: Psi definition})
\begin{equation*}\label{eq: Psi definition gauss}
\Psi(z): = \imag s(z) + \frac{p}{nv}.
\end{equation*}
To simplify notations we will often write $\Psi(v)$ and $\Psi$ instead of $\Psi(z)$.
\begin{lemma}\label{lemma: imag part of R_jj gauss}
Assuming conditions $\CondThree$ there exist a  positive constant $H_0$ depending on $u_0, V$ and positive constants $A_0, A_1$ depending on $H_0$ and $K$ such that for all $1 \le p \le A_1nv$ and $z \in \widetilde{\mathbb D}$  we get
$$
\max_{j \in \T }\E|\imag \RR_{jj}(z)|^p \le H_0^p \Psi^p(z).
$$
\end{lemma}
Note  that the values of $A_0$ and $A_1$ in this lemma are different from the values of corresponding quantities in Lemma~\ref{main lemma gauss}, but for simplicity we shall use the same notations. Applying both Lemmas we shall restrict the upper limit of the moment of order $p$ to the minimum of the two $A_1$'s and  the lower end of the range of $v$ to the maximum of the two $A_0$'s via $v \ge A_0n^{-1}$

For any $j \in \T_{\J}$ we may express $\RR_{jj}^{(\J)}$ in the following way
\begin{equation}\label{eq: representation for RR_jj}
\RR_{jj}^{(\J)} = \frac{1}{-z + \frac{X_{jj}}{\sqrt n} - \frac{1}{n}\sum_{l,k \in \T_{\J,j}} X_{jk} X_{jl} \RR_{lk}^{(\J,j)}}.
\end{equation}
Let $\varepsilon_j^{(\J)} : = \varepsilon_{1j}^{(\J)} + \varepsilon_{2j}^{(\J)}+\varepsilon_{3j}^{(\J)}+\varepsilon_{4j}^{(\J)}$, where
\begin{align*}
&\varepsilon_{1j}^{(\J)} =  \frac{1}{\sqrt n}X_{jj}, \quad \varepsilon_{2j}^{(\J)} = -\frac{1}{n}\sum_{l \ne k \in \T_{\J,j}} X_{jk} X_{jl} \RR_{kl}^{(\J,j)},
\quad \varepsilon_{3j}^{(\J)} = -\frac{1}{n}\sum_{k \in T_{\J,j}} (X_{jk}^2 -1) \RR_{kk}^{(\J)}(z), \\
&\varepsilon_{4j}^{(\J)}= \frac{1}{n} (\Tr \RR^{(\J)} - \Tr \RR^{(\J,j)}(z)).
\end{align*}
We also introduce the quantities $\Lambda_n^{(\J)}(z) : = m_n^{(\J)} (z) - s(z)$ and
$$
T_n^{(\J)}: = \frac{1}{n} \sum_{j \in \T_{\J}} \varepsilon_j^{(\J)}\RR_{jj}^{(\J)}.
$$
The following lemma, Lemma~\ref{lemma: recurence relation for R_jj}, allows  to  recursively estimate the moments of the diagonal entries of the resolvent.
 The proof of the first part of this lemma may be found in~\cite{Schlein2014} and
it is included here for the readers convenience.
\begin{lemma}\label{lemma: recurence relation for R_jj}
For an arbitrary set $\J \subset \T$ and all $j \in \T_\J$ there exist a positive constant $c_0$ depending on $u_0, V$ only such that for all $z = u + i v$ with $V \geq v > 0$ and $|u| \le u_0$ we have
\begin{equation}\label{inequality for R_jj}
|\RR_{jj}^{(\J)}| \le c_0\Big(1 + |T_n^{(\J)}|^{\frac{1}{2}}|\RR_{jj}^{(\J)}| + |\varepsilon_j^{(\J)}||\RR_{jj}^{(\J)}|\Big)
\end{equation}
and
\begin{equation}\label{inequality for 1/(z+m_n(z))}
\frac{1}{|z+m_n^{(\J)}(z)|} \le c_0\left(1 + \frac{|T_n^{(\J)}|^{\frac{1}{2}}}{|z+m_n^{(\J)}(z)|}\right).
\end{equation}
\end{lemma}
\begin{proof}
See~\cite{GotzeNauTikh2015a}[Lemma~4.2].
\end{proof}

\begin{lemma}\label{lemma step for resolvent}
		Assume that the conditions $\CondThree$ hold. Let $C_0$ and $s_0$ be arbitrary numbers such that $H_0 \geq \max(1/V, 6 c_0), s_0 \geq 2$. There exist a sufficiently large constant $A_0$ and small constant $A_1$ depending on $C_0, s_0, V$ only such that the following statement holds. Fix some $\tilde v: \tilde v_0 s_0 \leq \tilde v \le V$. Suppose that for some integer $L > 0$, all $u, v',q$ such that  $\tilde v \leq v' \leq V,\, |u| \le u_0, 1 \le q \le A_1 (n v')$
		\begin{equation}\label{main condition 1 0}
		\max_{\J: |\J| \le L} \max_{l, k \in \T_\J}\E |\RR_{l k}^{(\J)}(v')|^q \le C_0^q.
		\end{equation}
		Then for all $u,v, q$ such that $\tilde v/s_0 \leq v \le V, |u| \le u_0$, $1 \le q \le A_1 (nv)$
		\begin{equation*}\label{main condition 1 1}
		\max_{\J: |\J| \le L-1} \max_{l, k \in \T_\J}\E |\RR_{l k}^{(\J)}(v)|^q \le C_0^q.
		\end{equation*}
	\end{lemma}

\begin{proof}
Let us fix an arbitrary $s_0 \geq 2$ and $v \geq \tilde v/s_0$, $\J \subset \T$ such that $|\J| \le L-1$. In the following let $j, k \in \T_\J$. By an obvious inequality we have
\begin{equation*}
\E |\varepsilon_{j}^{(\J)}|^{2q} \le 3^{2q} (\E |\varepsilon_{j1}^{(\J)}|^{2q} + \E |\varepsilon_{j2}^{(\J)}+\varepsilon_{j3}^{(\J)}|^{2q} + \E|\varepsilon_{j4}^{(\J)}|^{2q}). 
\end{equation*}
From $\CondThree$ and Lemmas~\ref{appendix lemma varepsilon_2-3}--\ref{appendix lemma varepsilon_4} we may conclude that
\begin{align*}
\E |\varepsilon_{j}^{(\J)}|^{2q} \le \frac{C^q q^q}{n^q} + \frac{C^q q^q} {(nv)^q} \E \imag^q m_n^{(\J,j)}(z) + \frac{C^q q^{2q}}{(nv)^{2q}}.
\end{align*}
Applying Lemma~\ref{appendix lemma resolvent relations on different v} and~\eqref{main condition 1 0} we obtain the following estimate
\begin{equation*}\label{eq: bounds for RR jj 1}
\E[\imag m_n^{(\J,j)}(z)]^{q} \le s_0^{q} C_0^{q}.
\end{equation*}
In view of these inequalities we may write
\begin{align}\label{eq: eps_j est}
\E |\varepsilon_{j}^{(\J)}|^{2q} \le \frac{(C C_0 s_0)^q q^q } {(nv)^q} + \frac{C^q q^{2q}}{(nv)^{2q}}.
\end{align}
Similarly, we can estimate
\begin{align}\label{eq: T_n est}
\E |T_n^{(\J)}|^{q} &\le  \left(\frac{1}{n} \sum_{j \in \T_{\J}} \E |\varepsilon_j^{(\J)}|^{2q} \right)^{1/2} \left(\frac{1}{n} \sum_{j \in \T_{\J}} \E |\RR_{jj}^{(\J)}(v)|^{2q} \right)^{1/2} \nonumber \\
&\le  \frac{(C C_0^\frac32 s_0^\frac32)^q q^\frac{q}{2}} {(nv)^\frac q2} + \frac{(C C_0 s_0)^q q^{q}}{(nv)^{q}}.
\end{align}
Applying H\"older's inequality and Lemma~\ref{lemma: recurence relation for R_jj} we write 
\begin{align*}
\E|\RR_{jj}^{(\J)}(v)|^{q} \le (3 c)^{q}(1 + \E^{\frac{1}{2}}|T_n^{(\J)}|^{q}\E^{\frac{1}{2}}|\RR_{jj}^{(\J)}(v)|^{2q} + \E^{\frac{1}{2}}|\varepsilon_j^{(\J)}|^{2q} \E^{\frac{1}{2}}|\RR_{jj}^{(\J)}(v)|^{2q}).
\end{align*}
Lemma~\ref{appendix lemma resolvent relations on different v} and assumption~\eqref{main condition 1 0} imply
$\E|\RR_{jj}^{(\J)}(v)|^{2q} \le C_0^{2q} s_0^{2q}$. Hence, we may use~\eqref{eq: eps_j est} and~\eqref{eq: T_n est} to show that
\begin{align}\label{eq: R_j est}
\E|\RR_{jj}^{(\J)}(v)|^{q} \le (3 c)^{q} \left(1 +  \frac{(C C_0 s_0)^{2q} q^\frac{q}{4}} {(nv)^\frac q4} + \frac{(C C_0 s_0)^\frac{3q}{2} q^\frac q2}{(nv)^\frac q2}  + \frac{(C C_0 s_0)^q q^{q}}{(nv)^{q}} \right).
\end{align}
The off-diagonal entries $\RR_{jk}^{(\J)}$ may be expressed as follows
$$
\RR_{jk}^{(\J)} = -\frac{1}{\sqrt n} \sum_{l \in \T_{\J,j}} X_{jl} \RR_{lk}^{(\J,j)} \RR_{jj}^{(\J)}.
$$
Applying  H{\"o}lder's inequality we obtain
$$
\E |\RR_{jk}^{(\J)}|^{q} \le n^{-\frac q2} \E^{\frac{1}{2}} \left| \sum_{l \in \T_{\J,j}} X_{jl} \RR_{lk}^{(\J,j)} \right|^{2q}  \E^{\frac{1}{2}} |\RR_{jj}^{(\J)}|^{2q}.
$$
Khinchine's inequality, Lemmas~\ref{appendix lemma resolvent inequalities 1}--\ref{appendix lemma resolvent relations on different v} and assumption~\eqref{main condition 1 0} together imply 
\begin{equation}\label{eq: R_{jk}}
\E |\RR_{jk}^{(\J)}|^{q} \le \frac{(C C_0^\frac32 s_0^\frac32)^q q^\frac q2}{(nv)^\frac q2}.
\end{equation}
We may choose now  the constants $A_0$ sufficiently large, respectively  $A_1$ sufficiently small such that~\eqref{eq: R_j est}--\eqref{eq: R_{jk}} result in
$$
\max_{j,k \in \T_\J}\E|\RR_{jk}^{(\J)}(v)|^{q} \le  C_0^q.
$$
for $1 \le q \le A_1 (n \tilde v/s_0)$, $v \geq \tilde v/s_0$.
\end{proof}

\begin{proof}[Proof of Lemma~\ref{main lemma gauss}]
We first prove~\eqref{eq: main lemma first statement gauss}.  Let us choose some sufficiently large constant $C_0 > \max(1/V, 6 c_0)$ and fix $s_0: = 2$. Here $c_0$ is defined in Lemma~\ref{lemma: recurence relation for R_jj}. We also choose $A_0$ and $A_1$ as in Lemma~\ref{lemma step for resolvent}. Let $L:= [\log_{s_0} V/\tilde v_0]+1$. Since $\|\RR^{(\J)}(V)\| \le V^{-1}$ we may write 
$$
\max_{\J: |\J| \le  L} \max_{l,k \in \T_\J} \E|\RR_{lk}^{(\J)}(V)|^p \le C_0^p
$$
for all $u, p$ such that $|u| < 2$ and $1 \le p \le A_1 (nV)$. Fix arbitrary $v: V/s_0 \le  v \leq V$ and $p: 1 \le p \le A_1 (nv)$. Lemma~\ref{lemma step for resolvent} yields that
$$
\max_{\J: |\J| \le L-1} \max_{l, k \in \T_\J} \E|\RR_{lk}^{(\J)}(v)|^p \le C_0^p
$$
for $1 \le p \le A_1 (n V/s_0)$, $v \geq V/s_0$. We may repeat this procedure $L$ times and finally obtain
$$
\max_{l,k \in \T}\E|\RR_{lk}(v)|^p \le C_0^p
$$
for $1 \le p \le A_1 (n V /s_0^{L}) \le A_1 (n \tilde v_0)$ and $v \geq V/s_0^{L} = \tilde v_0$. Thus we proved~\eqref{eq: main lemma first statement gauss}. Similarly one may prove~\eqref{inequality for 1/(z+m_n(z))}.
\end{proof}

The following lemma is the analogue of Lemma~\ref{lemma: recurence relation for R_jj} and provides a recurrence relation for $\imag \RR_{jj}$.
\begin{lemma}\label{lemma recurence relation for imag R_jj}
For any set $\J$ and $j \in \T_\J$ there exists a positive constant $C_1$ depending on $u_0, V$ such that for all $z = u + iv$ with $V \geq v > 0$ and $|u| \le u_0$ we have
\begin{align*}
\imag \RR_{jj}^{(\J)}(z) &\le C_1 \left[\imag s(z) (1 + (|\varepsilon_j^{(\J)}| + |T_n^{(\J)}|^{\frac{1}{2}})|\RR_{jj}^{(\J)}(z)|)   +  |\imag \varepsilon_j^{(\J)} + \imag \Lambda_n^{(\J)}| |\RR_{jj}^{(\J)}(z)|\right.\\
&\left.\qquad\qquad + (|\varepsilon_j^{(\J)}| + |T_n^{(\J)}|^{\frac{1}{2}})\imag \RR_{jj}^{(\J)}(z) \right ].
\end{align*}
\end{lemma}
\begin{proof}
The proof is similar to the proof of Lemma~\ref{lemma: recurence relation for R_jj} is omitted.
\end{proof}

\begin{lemma}\label{lem: imag RR jj recur}
Assume that the conditions $\CondThree$ hold. Let $H_0$ be sufficiently large positive constant and $s_0$ be an arbitrary number such that $s_0 \geq 2$. There exist  sufficiently large  $A_0$ and small $A_1$ depending on $H_0, s_0, V$ only such that the following holds.  Fix some $\tilde v: \tilde v_0 s_0 \leq \tilde v \le V$. Suppose that for some integer $L > 0$, all $u, v',q$ such that  $\tilde v \leq v' \leq V,\, |u| \le u_0, 1 \le q \le A_1 (n v')$
\begin{equation}\label{main condition imag part of R_jj}
\max_{\J: |\J| \le L} \max_{l \in \T_\J} \E \imag^q \RR_{ll}^{(\J)}(v') \le H_0^q \Psi^q(v')
\end{equation}
Then for all $u,v, q$ such that $\tilde v/s_0 \leq v \le V, |u| \le u_0$, $1 \le q \le A_1 (nv)$
$$
\max_{\J: |\J| \le L-1} \max_{l \in \T_\J} \E \imag^q \RR_{ll}^{(\J)}(v) \le H_0^q \Psi^q(v).
$$
\end{lemma}
\begin{proof}
From Lemma~\ref{lemma recurence relation for imag R_jj} it follows that
\begin{align*}
\E\imag^q \RR_{jj}^{(\J)} &\le (C C_0)^q \imag^q s(z) \E^\frac12 (1 + (|\varepsilon_j^{(\J)}| + |T_n^{(\J)}|^{\frac{1}{2}})^{2q}  \nonumber \\
&\qquad + (C C_0)^q \E^{\frac{1}{2}}|\imag \varepsilon_j^{(\J)} + \imag \Lambda_n^{(\J)}|^{2q} \nonumber \\
&\qquad + C^q \E^{\frac{1}{2}}(|\varepsilon_j^{(\J)}| + |T_n^{(\J)}|^{\frac{1}{2}})^{2q} \E^{\frac{1}{2}} \imag^{2q} \RR_{jj}^{(\J)}.
\end{align*}
To estimate $\E |\varepsilon_j^{(\J)}|^{2q}$ and $\E|T_n^{(\J)}|^{q}$ we may proceed as in Lemma~\ref{lemma step for resolvent}. We obtain the following inequalities
\begin{align}\label{eq: imag R step 2}
\E |\varepsilon_j^{(\J)}|^{2q} &\le 3^{2q} \left[ \frac{C^q q^q}{n^q}  +  \frac{(C)^q q^q} {(nv)^q} + \frac{C^q q^{2q}}{(nv)^{2q}} + \frac{1}{(nv)^{2q}}\right] \le \frac{C^q q^q}{(nv)^q}
\end{align}
and
\begin{align}\label{eq: imag R step 3}
\E|T_n^{(\J)}|^{q} \le C_0^q \left(\frac{1}{n} \sum_{j \in \T_{\J}} \E |\varepsilon_j^{(\J)}|^{2q} \right)^{1/2} \le \frac{C^q q^\frac{q}{2}}{(nv)^\frac{q}{2}}
\end{align}
Choosing $v':=s_0 v \geq v_1$ we may show that $2q \le A_1 n v'$. Applying Lemma~\ref{appendix lemma imag resolvent relations on different v} and using the assumption~\eqref{main condition imag part of R_jj} we get
\begin{equation*}
\E \imag^{2q} \RR_{jj}^{(\J)}( v) \le s_0^{2q} \E \imag^{2q} \RR_{jj}^{(\J)}(s_0 v) \le s_0^{2q} H_0^{2q} \Psi^{2q}(s_0 v).
\end{equation*}
Since we need an estimate involving $\Psi^{2q}(v)$ instead of $\Psi^{2q}(s_0 v)$ on the r.h.s. of the previous inequality we need to perform a descent along the imaginary line from $s_0 v$ to $v$. Hence  we again need to apply Lemma~\ref{appendix lemma imag resolvent relations on different v}.  Choosing suitable constants $A_0$ and $A_1$ in~\eqref{eq: imag R step 2} and~\eqref{eq: imag R step 3} one may show that
\begin{align}\label{eq: imag R step 1 0}
\E\imag^q \RR_{jj}^{(\J)} &\le (C C_0)^q \E^{\frac{1}{2}}|\imag \varepsilon_j^{(\J)} + \imag \Lambda_n^{(\J)}|^{2q} + \frac{H_0^q}{2} \Psi^q.
\end{align}
Applying Lemmas~\ref{appendix lemma imag varepsilon_2-3} and~\ref{appendix lemma varepsilon_4} we obtain
$$
\E |\imag \varepsilon_j^{(\J)}|^{2q} \le \frac{C^q q^q}{(nv)^q} \E \imag^q m_n^{(\J,j)}(z) + \frac{C^q q^{2q}}{(nv)^{2q}}.
$$
which may be rewritten as follows
\begin{align}\label{eq: imag R step 4}
\E |\imag \varepsilon_j^{(\J)}|^{2q} &\le \frac{(C s_0)^{2q} q^q H^q}{(nv)^q} \Psi^q(z) + \frac{C^q q^{2q}}{(nv)^{2q}}.
\end{align}
To estimate $\E |\imag \Lambda_n^{(\J)}|^{q}$ we may proceed as in the proof of Theorem~\ref{th:main}. We will apply Theorem~\ref{th: general bound} (one has to replace in the definition of~\eqref{definition of A} the maximum over $|\J| \le 1$ by the maximum over $|\J| \le L$) and assumption~\eqref{main condition imag part of R_jj}. Hence,
\begin{align}\label{eq: imag R step 5}
\E |\imag \Lambda_n^{(\J)}|^{2q} \le \frac{(C s_0)^{2q} q^q H^q}{(nv)^q} \Psi^q(z) + \frac{C^q q^{2q}}{(nv)^{2q}}.
\end{align}
Combining the estimates~\eqref{eq: imag R step 4} and \eqref{eq: imag R step 5} we may choose constants $A_0$ and $A_1$ (correcting the previous choice if needed) such that
$$
(C C_0)^q \E^{\frac{1}{2}}|\imag \varepsilon_j^{(\J)}| \le \frac{H_0^q}{2}\Psi^q(z).
$$
The last two inequalities and~\eqref{eq: imag R step 1 0} together imply the desired bound
$$
\E \imag^q \RR_{jj}^{(\J)} \le H_0^q \Psi^q.
$$
\end{proof}
\begin{proof}[Proof of Lemma~\ref{lemma: imag part of R_jj gauss}]
Let us take any $u_0>0$ and any $\hat v \ge 2+u_0, |u|\le u_0$. Furthermore we fix an arbitrary $\J\subset \T$. We claim that
\begin{equation}\label{bound of imag RR via imag s}
\imag s(u+i\hat v)\ge \frac12\imag \RR_{jj}^{(\J)}(u+i\hat v).
\end{equation}
Indeed, note first that for all $u$ (and $|u|\le u_0$ as well)
\begin{equation}\label{bound for imag RR_{jj} for big V}
\imag \RR_{jj}^{(\J)}(u+i \hat v)\le \frac{1}{\hat v}.
\end{equation}
For all $|u|\le u_0$ and $|x|\le 2$ we obtain
\begin{equation*}
\frac{\hat v}{(x-u)^2+\hat v^2}\ge\frac{\hat v}{(2+u_0)^2+\hat v^2} \geq \frac{1}{2\hat v}.
\end{equation*}
It follows from the last inequality that
\begin{equation}\label{bound for imag s for big V}
\imag s(u+i \hat v)=\frac1{2\pi}\int_{-2}^2\frac{\hat v}{(u-x)^2+\hat v^2}\sqrt{4-x^2}dx\ge \frac{1}{2\hat v}.
\end{equation}
Comparing~\eqref{bound for imag RR_{jj} for big V} and~\eqref{bound for imag s for big V} we arrive at~\eqref{bound of imag RR via imag s}.

We now take $v \geq \max(\hat v, V)$. Let $H_0$ be some large constant, $H_0 \geq \max(C', C'')$. We choose $s_0, A_0$ and $A_1$  as in the previous Lemma~\ref{lem: imag RR jj recur} obtaining
$$
\max_{\J: |\J| \le  L} \max_{j \in \T_\J} \imag^q \RR_{jj}^{(\J)}(z) \le H_0^q  \Psi^q(z)
$$
with $L = [-\log_{s_0} \tilde v_0] + 1$. We may now proceed recursively  in $L$ steps
and arrive at
$$
\max_{j \in \T} \imag^q \RR_{jj}(z) \le H_0^q \Psi^q(z)
$$
for $v \geq \tilde v_0$ and  $1 \le q \le A_1 nv$ .
\end{proof}

\section{Delocalization of eigenvectors}\label{sec: delocalization}
In this section we prove Theorem~\ref{th: delocalization}. The ideas of the proof are similar to~\cite{GotzeNauTikh2015b}[Theorem~1.4], but for completeness we provide the details below. Note that the proof is essentially based on Lemma~\ref{main lemma}.
\begin{proof}[Proof of Theorem~\ref{th: delocalization}] Let us introduce the following distribution function
	$$
	F_{nj}(x): = \sum_{k=1}^n |u_{jk}|^2 \one[\lambda_k(\W)\le x].
	$$
	Using the eigenvalue decomposition of $\W$ it is easy to see that
	$$
	\RR_{jj}(z) = \sum_{k=1}^n\frac{|u_{jk}|^2}{\lambda_k(\W) - z} =  \int_{-\infty}^\infty \frac{1}{x - z} \, d F_{nj}(x),
	$$
	which means that $\RR_{jj}(z)$ is the Stieltjes transform of $F_{nj}(x)$. For any $\lambda > 0$ we have
	\begin{equation}\label{eq: Q nj}
	\max_{1 \le k \le n} |u_{jk}|^2 \le \sup_x (F_{nj}(x + \lambda) - F_{nj}(x)) \le 2 \sup_u \lambda \imag \RR_{jj}(u + i\lambda).
	\end{equation}
	To finish the proof we need to show that with high probability the r.h.s. of~\eqref{eq: Q nj} is bounded by $n^{-1} \log n$.
	Let us recall the  following notations.  We chose an arbitrary $0 < \phi' < \frac{1}{4}$. Let $\hat X_{jk}: = X_{jk} \one[|X_{jk}| \leq D n^{\frac12-\phi'}]$, $\tilde X_{jk}: = X_{jk} \one[|X_{jk}| \leq D n^{\frac12-\phi'}] - \E X_{jk} \one[|X_{jk}| \leq D n^{\frac12-\phi'}]$ and finally
	$\breve X_{jk}: = \tilde X_{jk} \sigma^{-1}$, where $\sigma^2: = \E |\tilde X_{11}|^2$.
	Let $\widehat \X, \widetilde \X$ and $\breve \X$ denote symmetric random matrices  with entries $\hat X_{jk}, \tilde X_{jk}$ and $\breve X_{jk}$ respectively. Similarly we denote the resolvent matrices by $\widehat \RR, \widetilde \RR$ and $\breve \RR$.
	In this case we have
	$$
	\Pb(\W \neq \widehat \W) \le \frac{C}{n^{2-\phi}},
	$$
	where $\phi: = 8 \phi'$. 
	Let $u_0 > 0$ denote  a large constant, whose exact value will be chosen later. Applying~\cite{GotzeNauTikh2015b}[Lemmas~A.1,~A.2] it follows  that
	\begin{equation*}
	\Pb(\|\W \| \geq u_0) \le \frac{C}{n^{2-\phi}}.
	\end{equation*}
	In what follows we may assume that $\|\W\| \le u_0$ and $\W = \widehat \W$. Then for $|u| \geq 2u_0$ and $v > 0$ we get
	\begin{equation*}
	|\RR_{jj}(u + i v)| \le \int_{-u_0}^{u_0} \frac{1}{\sqrt{(x-u)^2 + v^2}} \, dF_{nj}(x) \le \frac{1}{u_0} \le C,
	\end{equation*}
	where $C$ is some large positive constant which will be chosen later.
	It remains to estimate $|\RR_{jj}(u + i v)|$ for all $-2u_0 \le u \le 2u_0$. Denote this interval by  $\mathcal U_0: = [-2u_0, 2u_0]$.
	By the triangular inequality we may write $|\RR_{jj}| = |\widehat \RR_{jj}| \le |\widetilde \RR_{jj}| + |\widehat \RR_{jj} - \widetilde \RR_{jj}|$. Using  the simple identity
	$$
	\widehat \RR_{jj} - \widetilde \RR_{jj} = [\widehat \RR (\widehat \W - \widetilde \W) \widetilde \RR]_{jj}
	$$
	we get
	$$
	|\widehat \RR_{jj} - \widetilde \RR_{jj}| \le \| \widehat \W -\widetilde \W\| \|\ee_j^\mathsf{T} \widehat \RR\|_2 \|\widetilde \RR \ee_j\|_2,
	$$
	where $\ee_j$ is a unit column-vector with all entries zero except for an entry one at
	the position $j$. Using Lemma~\ref{appendix lemma resolvent inequalities 1} in the Appendix we conclude that
	$$
	|\widehat \RR_{jj}| \le  |\widetilde \RR_{jj}| + \frac{1}{v}  \|\widehat \W -\widetilde \W\|  \sqrt{|\widehat \RR_{jj}||\widetilde \RR_{jj}|}.
	$$
	It is easy to see that
	$$
	\|\widehat \W - \widetilde \W\|_2^2 = \frac{1}{n} \sum_{j,k} [\E |X_{jk}| \one[|X_{jk}| \geq D n^{\frac12-\phi}]]^2  \le \frac{C}{n^4},
	$$
	We may take $v = v_0: = C_1 n^{-1} \log n$, with $C_1 \geq A_0$. Applying the inequality $2 |a b| \le a^2 + b^2$ we get
	\begin{equation} \label{eq: bound for RR_jj via hat RR jj}
	\sup_{u \in \mathcal U_0}|\RR_{jj}| \le  3\sup_{u \in \mathcal U_0}|\widetilde \RR_{jj}|.
	\end{equation}
	It remains to estimate $\sup_{u \in \mathcal U_0} |\widetilde \RR_{jj}(u + i v_0)|$. It is easy to see that
	\begin{equation}\label{eq: tilde R jj representation deloc}
	\widetilde \RR(z) = (\widetilde \W - z\II)^{-1} = \sigma^{-1} (\breve \W - z \sigma^{-1}\II)^{-1}  = \sigma^{-1}\breve \RR(\sigma^{-1}z).
	\end{equation}
	Applying the resolvent identity we get
	\begin{equation}\label{eq: overline R jj representation deloc}
	\breve \RR(z) - \breve \RR(\sigma^{-1}z) = (z - \sigma^{-1}z) \breve \RR(z) \breve \RR(\sigma^{-1}z).
	\end{equation}
	Combining~\eqref{eq: tilde R jj representation deloc} and~\eqref{eq: overline R jj representation deloc} we obtain
	\begin{align*}
	|\widetilde \RR_{jj}(z) - \breve \RR_{jj}(z)| \le (\sigma^{-1} - 1) |\breve \RR_{jj}(\sigma^{-1} z)| + \frac{|z|(\sigma^{-1} - 1)}{v} \sqrt{|\breve \RR_{jj}(z)| |\breve \RR_{jj}(\sigma^{-1}z)|}.
	\end{align*}
	It is easy to check that $(\sigma^{-1} - 1) \le C n^{-\frac32}$ and $\max(|z \breve \RR_{jj}(z)|, |z \breve \RR_{jj}(\sigma^{-1}z)|) \le C$ for some constant $C$.
	Similarly to the previous calculations we get that
	\begin{equation} \label{eq: bound for tilde RR_jj via breve RR jj}
	\sup_{u \in \mathcal U_0}|\widetilde \RR_{jj}| \le  3\sup_{u \in \mathcal U_0}|\breve \RR_{jj}|.
	\end{equation}
	Note, that the matrix $\breve \W$ satisfies the conditions $\CondTwo$. Applying Lemma~\ref{main lemma} with $p = c \log n$ we obtain
	$$
	\Pb(|\breve \RR_{jj}(u+ i v_0)| \geq C_0 e^\frac{6}{c}) \le \frac{\E |\breve \RR_{jj}(u+i v_0)|^p}{(C_0 e^\frac{6}{c})^p} \le \frac{1}{n^{6}}.
	$$
	We  partition the interval  $\mathcal U_0$ into $k_n : = n^4$ disjoint sub-intervals of equal length, i.e $-2u_0 = x_0 \le x_1 \le ... \le x_{k_n} = 2 u_0$. Then the Newton-Leibniz formula implies
	\begin{align*}
	\sup_{u \in \mathcal U_0} |\breve \RR_{jj}(u + i v_0)| &\le \max_{1\le k \le k_n} \sup_{x_{k-1} \le x \le x_k}|\breve \RR_{jj}(x + i v_0)|\\
	& \le  \max_{1\le k \le k_n} |\breve \RR_{jj}(x_{k-1} + i v_0)| + \max_{1\le k \le k_n} \int_{x_{k-1}}^{x_k} |\breve \RR_{jj}'(u+iv_0)| \, du.
	\end{align*}
	We may write
	$$
	\max_{1\le k \le k_n} \int_{x_{k-1}}^{x_k} |\breve \RR_{jj}'(u+iv_0)| \, du \le \frac{C}{n^{2-\phi}}.
	$$
	Thus we arrive at
	\begin{align}\label{eq: bound for probability of breve RR jj}
	&\Pb\left(\sup_{u \in \mathcal U_0} |\breve \RR_{jj}(u + iv_0)| \geq 2C_0 e^\frac{6}{c} \right) \nonumber\\
	&\qquad\qquad\qquad\qquad\le \sum_{k=1}^{k_n} \Pb \left (|\breve \RR_{jj}(x_{k-1} + iv_0)| \geq C_0 e^\frac{6}{c}\right) \le \frac{C }{n^{2}}.
	\end{align}
	We choose now $\lambda: = v_0$. In view of~\eqref{eq: Q nj},~\eqref{eq: bound for RR_jj via hat RR jj},~\eqref{eq: bound for tilde RR_jj via breve RR jj} and~\eqref{eq: bound for probability of breve RR jj} we get that there exist $C$ and $C_1$ such that
	$$
	\Pb \left(\max_{1 \le j, k \le n} |u_{jk}|^2 \leq \frac{C_1 \log n}{n} \right) \geq 1 - \frac{C}{n^{2-\phi}},
	$$
	which concludes the proof.
\end{proof}

\section{Rate of convergence to the semicircle law}\label{sec: rate of convergence}
In this section we prove Theorem~\ref{th: rate of convergence} and Theorem~\ref{th: rate of conv expected}. We estimate the difference between $F_n$ (resp. $\E F_n$) and $G_{sc}$ in the Kolmogorov metric via the distance between the corresponding Stieltjes transforms. For this purpose we formulate the following smoothing inequality proved in~\cite{GotTikh2003}[Corollary~2.3], which allows to relate   distribution functions to their Stieltjes transforms. For all $x \in [-2, 2]$ let us define $\gamma(x): = 2 -|x|$. Given $\frac{1}{2} > \varepsilon > 0$ we introduce the following intervals $\mathbb J_\varepsilon: = \{x \in [-2, 2]: \gamma(x) \geq \varepsilon\}$ and $\mathbb J_\varepsilon^{'} : =\mathbb J_{\varepsilon/2}$.
\begin{lemma} \label{l: bound for delta}
Let $v_0 > 0$ and $\frac{1}{2} > \varepsilon > 0$ be positive numbers such that
$$
2 (\sqrt 2 + 1) v_0  \le \varepsilon^{\frac32}.
$$
Assume that $F$ is an arbitrary distribution function with the Stieltjes transform $f(z)$. Then for any $V > 0$ and $v' := v'(x): =  v_0/\sqrt{\gamma(x)}, x \in \mathbb J_\varepsilon^{'}$, there exist absolute positive constants $C_1, C_2$ and $C_3$ such that the following inequality holds
\begin{align*}
\Delta(F,G) &\le C_1\int_{-\infty}^\infty |f(u + i V) - s(u + i V)|\, du + C_2 v_0 + C_3 \varepsilon^\frac32 \\
&\qquad\qquad\qquad+2\sup_{x \in \mathbb J_\varepsilon^{'}} \left | \int_{v'}^V (f(x + i v) - s(x + i v)) \, dv \right |.
\end{align*}
\end{lemma}
\begin{proof}
See~\cite{GotTikh2003}[Corollary~2.3] or~\cite{GotzeTikh2014rateofconv}[Proposition~2.1].
\end{proof}
In what follows we will need the following version of this lemma.
\begin{corollary}\label{cor: bound for delta}
Assuming the conditions of Lemma~\ref{l: bound for delta} we have
\begin{align}\label{eq: smoothing inequality}
\E^\frac{1}{p}[\Delta_n^{*}]^p &\le C_1\int_{-\infty}^\infty \E^\frac{1}{p}|m_n(u + i V) - s(u + i V)|^p\, du + C_2 v_0 + C_3 \varepsilon^\frac32 \nonumber \\
&\qquad\qquad\qquad+C_1\E^\frac{1}{p} \sup_{x \in \mathbb J_\varepsilon^{'}} \left | \int_{v'}^V (m_n(x + i v) - s(x + i v)) \, dv \right |^p.
\end{align}
\end{corollary}
\begin{proof}
The proof is the direct consequence of the previous lemma and we omit it. For details the interested reader is referred to~\cite{GotzeTikh2014rateofconv}[Corollary 2.1].
\end{proof}
\begin{proof}[Proof of Theorem~\ref{th: rate of convergence}] We start from the part $(i)$. We proceed as in the proof of Theorem~1.1 in~\cite{GotzeTikh2014rateofconv}. We choose in Corollary~\ref{cor: bound for delta} the following values for the parameters $v_0, \varepsilon$ and $V$. Let us take $v_0: = A_0 n^{-1} \log n$, $\varepsilon: = (2 v_0 a)^\frac23$ and $V: = 4$.
We may partition $\mathbb J_\varepsilon^{'}$ into $k_n := n^4$ disjoint subintervals of equal length. Let us denote the endpoints of these intervals by $x_k, k = 0, ... , k_n$. We get $-2 + \varepsilon = x_0 < x_1 < ... < x_{k_n} = 2 - \varepsilon$. For simplicity we denote $\Lambda_n(u+iv): = m_n(u+iv) - s(u+iv)$ but we will not omit the argument.
We start to estimate the second integral in the r.h.s. of~\eqref{eq: smoothing inequality}. It is easy to see that
\begin{align}\label{eq: second integral 0}
&\sup_{x \in \mathbb J_\varepsilon^{'}} \left| \int_{v'}^V \Lambda_n(x + i v) \, dv \right | \le \max_{1 \le k \le k_n} \sup_{x_{k-1} \le x \le x_k} \left| \int_{v'}^V \Lambda_n(x + i v) \, dv \right|.
\end{align}
Applying the Newton-Leibniz formula we may write
\begin{align}\label{eq: second integral 1}
\sup_{x_{k-1} \le x \le x_k} \left| \int_{v'}^V \Lambda_n(x + i v) \, dv \right| &\le \left| \int_{v'}^V \Lambda_n(x_{k-1} + i v) \, dv \right| \nonumber\\
&+ \int_{x_{k-1}}^{x_k}  \int_{v'}^V |\Lambda_n'(x + i v)| \, dv \,dx.
\end{align}
It follows from Cauchy's integral formula that for all $z = x + i v$ with $v \geq v_0$  we have
\begin{equation}\label{eq: second integral 2}
|\Lambda_n'(x + i v)| \le \frac{C}{v^2} \le C n^2.
\end{equation}
We may conclude from~\eqref{eq: second integral 1} and~\eqref{eq: second integral 2} that
\begin{align*}
&\sup_{x_{k-1} \le x \le x_k} \left| \int_{v'}^V \Lambda_n(x + i v) \, dv \right| \le \left| \int_{v'}^V \Lambda_n(x_{k-1} + i v) \, dv \right| + \frac{C}{n}.
\end{align*}
Applying this inequality to~\eqref{eq: second integral 0} together with expectations we obtain
\begin{align}\label{eq: second integral 3}
\E\sup_{x \in \mathbb J_\varepsilon^{'}} \left| \int_{v'}^V \Lambda_n(x + i v) \, dv \right |^p &\le \E \max_{1 \le k \le k_n} \left| \int_{v'}^V \Lambda_n(x_{k-1} + i v) \, dv \right|^p + \frac{C^p}{n^p}\nonumber \\
&\le \sum_{k=1}^{k_n} \left|  \int_{v'}^V \E^\frac{1}{p}\big|\Lambda_n(x_{k-1} + i v)\big|^p \, dv \right|^p + \frac{C^p}{n^p}.
\end{align}
Since $x \in \mathbb J_\varepsilon^{'}$ it follows from Theorem~\ref{th:main} that
\begin{align}\label{eq: bound for difference of ST}
\E|\Lambda_n(x + i v)|^p \le \left(\frac{Cp}{nv}\right)^p.
\end{align}
Choosing $p = A_1(n v_0)^{\frac{1-2\alpha}{2}} = c \log n$ we finally get from~\eqref{eq: second integral 3} and~\eqref{eq: bound for difference of ST} that
\begin{align}\label{eq: estimate 1}
&\E^\frac{1}{p} \sup_{x \in \mathbb J_\varepsilon^{'}} \left| \int_{v'}^V \Lambda_n(x + i v) \, dv \right |^p \le \frac{C k_n^\frac{1}{p} \log^2 n }{n}   + \frac{C}{n} \le \frac{ C \log^2 n}{n}.
\end{align}
It remains to estimate the first of the integrals in~\eqref{eq: smoothing inequality}. It was proved in~\cite{GotzeNauTikh2015b}[Inequality~2.8] that
\begin{align}\label{eq: estimate 2 assumption}
 \E^\frac{1}{p}|\Lambda_n(u + i V)|^p \le \frac{Cp |s(z)|^\frac{p+1}{p}}{n},
\end{align}
which holds for all $z = u + i V, u \in \R$. Hence,
\begin{align}\label{eq: estimate 2}
&\int_{-\infty}^\infty \E^\frac{1}{p}|\Lambda_n(u + i V)|^p\, du \le \frac{C p}{n} \int_{-\infty}^\infty \int_{-\infty}^\infty \frac{d u \,d G_{sc}(x)}{((x-u)^2 + V^2)^\frac{p+1}{p}} \le \frac{C\log^2 n}{n}.
\end{align}
Combining now~\eqref{eq: smoothing inequality},~\eqref{eq: estimate 1} and~\eqref{eq: estimate 2}  we get
$$
\E^\frac{1}{p}[\Delta_n^{*}]^p \le \frac{C\log^2 n}{n}.
$$
Since $\E^\frac{1}{p}[\Delta_n^{*}]^p$ is non-decreasing function of $p$, the last inequality remains valid for all $1 \le p \le c \log n$. To finish the proof of Theorem~\ref{th: rate of convergence} it remains to apply Markov's inequality
$$
\Pb\left(\Delta_n^{*} \geq K \right) \le  \frac{\E [\Delta_n^{*}]^p }{K^p} \le \frac{C^p \log^{2p} n}{K^p n^p}.
$$

\end{proof}

\begin{proof}[Proof of Theorem~\ref{th: rate of conv expected}]
Applying Lemma~\ref{l: bound for delta} with $F: = \E F_n$ we get
\begin{align}\label{smothing}
\Delta_n &\le 2\int_{-\infty}^\infty |\E m_n(u + i V) - s(u + i V)|\, du + C_1 v_0 + C_2 \varepsilon^\frac32 \nonumber \\
&\qquad\qquad\qquad+2\sup_{x \in \mathbb J_\varepsilon^{'}} \left | \int_{v'}^V (\E m_n(x + i v) - s(x + i v)) \, dv \right |.
\end{align}
Let us take $v_0: = A_0 n^{-1}$, $\varepsilon: = (2 v_0 a)^\frac23$ and $V: = 4$. As before  we denote $\Lambda_n(u+iv): = m_n(u+iv) - s(u+iv)$. The bound for the first integral follows from~\cite{GotTikh2015}[Inequality~3.11]
$$
|\E \Lambda_n(u + i V)| \le \frac{C |s(z)|^2}{n}. 
$$
This bound gives 
\begin{equation}
\int_{-\infty}^\infty |\E \Lambda_n(u + i V)|\, du \le \frac{C}{n}.
\end{equation}
To estimate the second integral we shall use the same arguments as in the proof of Lemma~\ref{lem: bound in the optimal region}.
We denote by $Y_{jk}, 1 \le j \le k \le n$ a triangular set of random variables such that $|Y_{jk}| \le D$, for some $D$ chosen later, and
$$
\E X_{jk}^s = \E Y_{jk}^s \, \text { for } \, s = 1, ... , 4.
$$
By Lemma~\ref{lem: random variables} these random variables exist. Let us denote $\W^\y: = \frac{1}{\sqrt n} \Y, \RR^\y: = (\W^\y - z \II)^{-1}$ and $m_n^\y(z): = \frac{1}{n} \Tr \RR^\y(z)$. We will show below that for all $j, j = 1, ... , n$, there exists $\phi > 0$ such that
\begin{equation}\label{eq: resolvent comparison}
|\E \RR_{jj}(z) - \E \RR_{jj}^\y(z)| \le \frac{C}{n^{1+\phi} v}
\end{equation}
and, hence, 
\begin{equation}\label{eq: stieltjes comparison}
|\E m_n(z) - \E m_n^\y(z)| \le \frac{C}{n^{1+\phi} v}.
\end{equation}
It follows from~\cite{GotTikh2015}[Theorem~1.3] that
\begin{equation}\label{eq: stieltjes sub gaussian}
|\E m_n^\y(z) - s(z)| \le \frac{C}{n v^\frac34} + \frac{C}{n^\frac32 v^\frac32 |z^2-4|^\frac14}.
\end{equation}
Inequalities~\eqref{eq: stieltjes sub gaussian} and~\eqref{eq: stieltjes comparison} together imply that
\begin{equation*}\label{eq: stieltjes general}
|\E m_n(z) - s(z)| \le \frac{C}{n v^\frac34} + \frac{C}{n^{1+\phi} v} + \frac{C}{n^\frac32 v^\frac32 |z^2-4|^\frac14}.
\end{equation*}
Calculating the second integral in~\eqref{smothing} we get
$$
\Delta_n \le \frac{C}{n}.
$$
It remains to prove~\eqref{eq: resolvent comparison}. For every $j, j = 1, ... , n$, we do $\frac{n(n+1)}{2}$ replacements and corresponding exchanges of $\RR_{jj}$ by $\RR_{jj}^\y$. Let $\J, \K \subset \T$. We denote by $\W^{(\J, \K)}$ a random matrix $\W$ 
with entries $\frac{1}{\sqrt n} Y_{\mu \nu}$
in the positions $(\mu, \nu), \mu \in \J, \nu \in \K $. Assume that we have already replaced the entries in positions $(\mu, \nu), \mu \in \J, \nu \in \K$ and want to replace in addition the entry in position $(a, b), a \in \T \setminus \J, b \in \T \setminus \K$. Without loss of generality we may assume that $\J = \emptyset, \K = \emptyset$ (hence $\W^{(\J, \K)} = \W$) and denote by $\V: = \W^{(\{a\}, \{b\})}$. Introduce
$$
\EE^{(a,b)} =
\begin{cases}
\ee_{a} \ee_b^\mathsf{T} + \ee_b \ee_a^\mathsf{T}, & 1 \le a < b \le n, \\
\ee_{a} \ee_a^\mathsf{T}, &a = b.
\end{cases}
$$
and $\U : = \W - \frac{X_{ab}}{\sqrt n}  \EE^{(a,b)}$, where $\ee_j$ denotes a unit column-vector with all entries zeros except in the $j$-th position. Using these notations we may write
$$
\W = \U + \frac{1}{\sqrt n} X_{ab} \EE^{(a,b)}, \quad \V = \U + \frac{1}{\sqrt n} Y_{ab} \EE^{(a,b)}.
$$
Recall that $\RR: = (\W - z \II)^{-1}$ and denote $\RS: = (\V - z \II)^{-1}$ and $\RT: = (\U - z \II)^{-1}$. Applying~\eqref{resolvent expansion R} and~\eqref{resolvent expansion S} we get 
\begin{align}\label{resolvent expansion R 2}
&\E\RR_{jj} - \E \RS_{jj} = \sum_{\mu=5}^m \frac{(-1)^\mu}{n^\frac{\mu}{2}} \E [X_{ab}^\mu - Y_{ab}^\mu]\E [(\RT \EE^{(a,b)})^\mu \RT]_{jj} \nonumber \\
&+ \frac{(-1)^{m+1}}{n^\frac{m+1}2} \E X_{ab}^{m+1} [(\RT \EE^{(a,b)})^{m+1} \RR]_{jj} - \frac{(-1)^{m+1}}{n^\frac{m+1}{2}} \E Y_{ab}^{m+1} [(\RT \EE^{(a,b)})^{m+1} \RS]_{jj}.
\end{align}
Without loss of generality we may assume that $|X_{jk}| \le D n^{\frac12 - \phi}$ for some $\phi > 0$ depending on $\alpha$.
Choosing $m$ and applying Lemma~\ref{main lemma non optimal} we estimate the third term in~\eqref{resolvent expansion R 2}  as follows
$$
\frac{1}{n^\frac{m+1}2} \E |X_{ab}|^{m+1} |[(\RT \EE^{(a,b)})^{m+1} \RR]_{jj}| \le \frac{C}{n^{3+\phi} v}.
$$
The same bound obviously holds for the fourth term in~\eqref{resolvent expansion R 2}. 
Let us consider now the first term and investigate the part corresponding to $X_{ab}$ (the same estimates are valid for the part corresponding to $Y_{ab}$) 
$$
\sum_{\mu=5}^m \frac{(-1)^\mu}{n^\frac{\mu}{2}} \E [X_{ab}^\mu]\E [(\RT \EE^{(a,b)})^\mu \RT]_{jj}
$$
It is straightforward to check that $[(\RT \EE^{(a,b)})^{\mu} \RT]_{jj}$ is the sum of $2^{\nu}$ terms of the following type
$$
\RT_{j i_1} \RT_{i_1 i_2} ... \RT_{i_m i_m} \RT_{i_m j},
$$
where $i_l = a$ or $i_l = b$ for $l = 1, ... , m$. Assume that $a$ and $b$ are not equal to $j$. Then, the first and the last terms in the last product are off diagonal entries of the resolvent $\RT$. Applying H{\"o}lder's inequality, Lemma~\ref{l: off diagonal entries} and Lemma~\ref{main lemma non optimal} we get  
$$
\sum_{\mu=5}^m \frac{1}{n^\frac{\mu}{2}} \E |X_{ab}|^\mu \E |[(\RT \EE^{(a,b)})^\mu \RT]_{jj}| \le \frac{C}{n^{3+\phi} v}.
$$
Assume now that $a$ or (and) $b$ are equal to $j$. The number of configurations of this type is  of order $n$. Similarly to the previous inequality we thus obtain the bound
$$
\sum_{\mu=5}^m \frac{1}{n^\frac{\mu}{2}} \E |X_{ab}|^\mu \E |[(\RT \EE^{(a,b)})^\mu \RT]_{jj}| \le \frac{C}{n^{2+\phi}}.
$$
Repeating the same steps for all $\frac{n(n+1)}{2}$ pairs of $(a,b)$ we arrive at~\eqref{eq: resolvent comparison} and~\eqref{eq: stieltjes comparison} respectively.
\end{proof}

\appendix
\section{Auxiliary lemmas}
\subsection{Inequalities for resolvent matrices}
In this section we collect some inequalities for the resolvent of the matrix $\W$.
\begin{lemma}\label{appendix lemma resolvent relations on different v}
For any $z = u + i v \in \C^{+}$ we have for any $s \geq 1$
\begin{equation*}
|\RR_{jj}^{(\J)}(u + i v/s)| \le s |\RR_{jj}^{(\J)}(u + i v)|.
\end{equation*}
and
\begin{equation}\label{appendix eq resolvent relations on different v 2}
\frac{1}{|u+iv/s_0 + m_n^{(\J)}(u+iv/s_0)|} \le \frac{s_0}{|u+iv + m_n^{(\J)}(u+iv)|}.
\end{equation}
\end{lemma}
\begin{proof}
See~\cite{GotzeNauTikh2015a}[Lemma~C.1].
\end{proof}

\begin{lemma}\label{appendix lemma imag resolvent relations on different v}
Let $g(v): = g(u+iv)$ be the Stieltjes transform of some distribution function $G(x)$. Then for any $s \geq 1$
\begin{align*}
\imag g(v/s) \le s \imag g(v) \, \text{ and } \, \imag g(v) \le s \imag g(v/s).
\end{align*}
\end{lemma}
\begin{proof}
See~\cite{GotzeNauTikh2015a}[Lemma~C.2].
\end{proof}

\begin{lemma}\label{appendix lemma resolvent inequalities 1}
For any $z = u + i v \in \C^{+}$ we have
\begin{equation}\label{appendix lemma resolvent inequality 1}
\frac{1}{n} \sum_{l,k \in \T_{\J}} |\RR_{kl}^{(\J)}|^2 \le \frac{1}{v} \imag m_n^{(\J)}(z).
\end{equation}
For any $l \in \T_{\J}$
\begin{equation}\label{appendix lemma resolvent inequality 2}
\sum_{k \in \T_{\J}} |\RR_{kl}^{(\J)}|^2 \le \frac{1}{v} \imag \RR_{ll}^{(\J)}.
\end{equation}
\end{lemma}
\begin{proof}
See~\cite{GotzeNauTikh2015a}[Lemma~C.4].
\end{proof}

\subsection{Moment inequalities for linear and quadratic forms in the sub-gaussian case} In this subsection we estimate the moments of $\varepsilon_{\nu j}^{(\J)}$ for $\nu = 1, ..., 4$ (recall the definition~\eqref{eq: representation for RR_jj}) in the sub-gaussian case. It is well-known that the random variables $\xi$ is sub-gaussian if and only if $\E |\xi|^p  = O(p^\frac{p}{2})$ as $p \rightarrow \infty$. We define the sub-gaussian norm of $\xi$ as
$$
\|\xi\|_{\psi_2}: = \sup_{p \geq 1} p^{-\frac12} \E^\frac{1}{p}|\xi|^p.
$$
We may conclude that $\E|\varepsilon_{1j}^{(\J)}|^p \le C p^\frac{p}{2} n^{-\frac{p}{2}}$.  The following lemma is the Hanson-Wright inequality for quadratic forms, see~\cite{HansonWright}. 
The following improved version is due to M. Rudelson and R. Vershynin~\cite{RudelsonVershynin2013}.

\begin{lemma}[Hanson-Wright inequality]\label{lem:  Hanson-Wright inequality} Let $X = (X_1, ..., X_n) \in \R^n$ be a random vector with independent components $X_i$ which satisfy $\E X_i = 0$ and $\|X\|_{\psi_2} \le K$. Let $\A = [a_{jk}]_{j,k=1}^n$ be an $n \times n$ matrix. Then, for every $t \geq 0$
$$
\Pb\left(\left|\sum_{j,k = 1}^n a_{jk} X_j X_k - \sum_{j=1}^n a_{jj} \E X_{jj}^2 \right| \geq t\right) \le 2 \exp \left[ -c \min\left(\frac{t^2}{K^4\|\A\|_2^2}, \frac{t}{K^2\|\A\|} \right) \right].
$$
\end{lemma}
\begin{proof}
See~\cite{RudelsonVershynin2013}[Theorem~1.1].
\end{proof}
A direct consequence of this lemma is the following result.
\begin{lemma}\label{appendix lemma varepsilon_2-3}
Assume that $X_1, ... , X_n$ are i.i.d. sub-gaussian random variables which satisfy $\E X_i = 0$ and $\|X\|_{\psi_2} \le K$. Then for all $p \geq 2$ there exists a positive constant $C$ depending on $K$ such that
$$
\E|\varepsilon_{2j}^{(\J)}+\varepsilon_{3j}^{(\J)}|^p \le \frac{C^p p^\frac{p}{2}} {(nv)^\frac{p}{2}} \E \imag^\frac{p}{2} m_n^{(\J,j)}(z) + \frac{C^p p^p}{(nv)^p}.
$$
\end{lemma}
\begin{proof}
Applying $\E|\xi|^p = p \int_{0}^\infty t^{p-1} \Pb (|\xi| \geq t) \, dt$ and Lemma~\ref{lem:  Hanson-Wright inequality}  we get
$$
\E|\varepsilon_{2j}^{(\J)}+\varepsilon_{3j}^{(\J)}|^p \le \frac{2 p}{n^p} \int_{0}^\infty t^{p-1} 2 \exp \left[ -c \min\left(\frac{t^2}{K^4\|\RR^{(\J,j)}\|_2^2}, \frac{t}{K^2\|\RR^{(\J,j)}\|} \right) \right] \, dt
$$
We may split the last integral in two integrals  over the regions $[0, \|\RR^{(\J,j)}\|_2^2 \|\RR^{(\J,j)}\|^{-1}]$ and $[\|\RR^{(\J,j)}\|_2^2 \|\RR^{(\J,j)}\|^{-1}, \infty]$ obtaining
\begin{align*}
\E|\varepsilon_{2j}^{(\J)}+\varepsilon_{3j}^{(\J)}|^p &\le \frac{2 p \|\RR^{(\J,j)}\|_2^p}{n^p}  \int_0^\infty t^p e^{-c \frac{t^2}{K^4}} \, dt + \frac{2 p \|\RR^{(\J,j)}\|^p}{n^p} \int_0^\infty t^p e^{-c \frac{t}{K^2}} \, dt \\
&\le \frac{C^p p^\frac{p}{2}} {(nv)^\frac{p}{2}} \E \imag^\frac{p}{2} m_n^{(\J,j)}(z) + \frac{C^p p^p}{(nv)^p},
\end{align*}
where in the second inequality we applied Lemma~\ref{appendix lemma resolvent inequalities 1} and $ \|\RR^{(\J,j)}\|^p \le v^{-p}$.
\end{proof}
It remains to estimate $\E |\varepsilon_{4j}^{(\J)}|^p$. The bound for this term is distribution-independent and given in the following lemma.
\begin{lemma}\label{appendix lemma varepsilon_4}
For $p \geq 2$ we have
$$
\E |\varepsilon_{4j}^{(\J)}|^p \le \frac{1}{(nv)^p}.
$$
\end{lemma}
\begin{proof}
See~\cite{GotzeNauTikh2015a}[Lemma~A.9].
\end{proof}
In order to bound $\E \imag^p \RR_{jj}$ we also need the following lemma.
\begin{lemma}\label{appendix lemma imag varepsilon_2-3}
Assume that $X_1, ... , X_n$ are i.i.d. sub-gaussian random variables which satisfy $\E X_i = 0$ and $\|X\|_{\psi_2} \le K$. Then for all $p \geq 2$ there exists a  positive constant $C$ depending on $K$ such that
$$
\E|\imag \varepsilon_{2j}^{(\J)}+\imag\varepsilon_{3j}^{(\J)}|^p \le \frac{C^p p^\frac{p}{2}} {(nv)^\frac{p}{2}} \E \imag^\frac{p}{2} m_n^{(\J,j)}(z) + \frac{C^p p^p}{(nv)^p}.
$$
\end{lemma}
\begin{proof}
The proof is similar to the proof of the previous Lemma~\ref{appendix lemma varepsilon_2-3}. We omit the details.
\end{proof}

\subsection{Moment inequalities for off-diagonal entries of  the resolvent}
\begin{lemma} \label{l: off diagonal entries}
Assume that the conditions $\CondTwo$ hold. Then for all $1 \le j < k \le n$ and $1 \le q \le 2 + \frac{\delta}{2}$ there exists a positive constant $C$ such that
$$
\E |\RR_{jk}|^q \le \frac{C}{(nv)^\frac{q}{2}}.
$$
\end{lemma}
\begin{proof}
Consider the following equality
\begin{equation*}
\RR_{jk}=-\frac1{\sqrt n}\left(\sum_{l\in \T_j} X_{jl} \RR^{(j)}_{lk}\right)\RR_{jj}.
\end{equation*}
Applying H{\"o}lder's inequality, we get
\begin{equation*}
\E|\RR_{jk}|^q\le n^{-\frac q2} \E^{\frac12}\left|\sum_{l\in \T_j} X_{jl} \RR^{(j)}_{lk}\right|^{2q}\E^{\frac12}|\RR_{jj}|^{2q}.
\end{equation*}
Conditioning, applying Rosenthal's inequality and Lemma~\ref{main lemma non optimal} we obtain
\begin{equation*}
\E|\RR_{jk}|^q\le C^q n^{-\frac q2}\left(q^{\frac q2}\E^{\frac12}\left(\sum_{l\in\T_j}|\RR^{(j)}_{lk}|^2\right)^{q}+q^{q}\mu_{2q}^{\frac12}\E^{\frac12}\left(\sum_{l\in\T_j}|\RR^{(j)}_{lk}|^{2q}\right)\right).
\end{equation*}
This inequality and Lemma~\ref{appendix lemma resolvent inequalities 1} together imply that
\begin{equation*}
\E|\RR_{jk}|^q\le C^q q^{\frac q2}(nv)^{-\frac q2}+q^q\mu_{2q}^{\frac12}(nv)^{-\frac q2}.
\end{equation*}
Hence, for $1 \le q \le 2+\frac{\delta }2$ we get 
\begin{equation*}
\E|\RR_{jk}|^q \le C(nv)^{-\frac{q}{2}}.
\end{equation*}
\end{proof}

\bibliographystyle{plain}
\bibliography{literatur}

\def\polhk#1{\setbox0=\hbox{#1}{\ooalign{\hidewidth
  \lower1.5ex\hbox{`}\hidewidth\crcr\unhbox0}}}
\begin{thebibliography}{10}

\bibitem{AndersonZeit}
G.~Anderson, A.~Guionnet, and O.~Zeitouni.
\newblock {\em An introduction to random matrices}, volume 118 of {\em
  Cambridge Studies in Advanced Mathematics}.
\newblock Cambridge University Press, Cambridge, 2010.

\bibitem{Arnold1967}
L.~Arnold.
\newblock On the asymptotic distribution of the eigenvalues of random matrices.
\newblock {\em J. Math. Anal. Appl.}, 20:262--268, 1967.

\bibitem{Bai1993}
Z.~Bai.
\newblock Convergence rate of expected spectral distributions of large random
  matrices. {I}. {W}igner matrices.
\newblock {\em Ann. Probab.}, 21(2):625--648, 1993.

\bibitem{BaiHuPanZhou2011}
Z.~Bai, J.~Hu, G.~Pan, and W.~Zhou.
\newblock A note on rate of convergence in probability to semicircular law.
\newblock {\em Electron. J. Probab.}, 16:no. 88, 2439--2451, 2011.

\bibitem{BaiMiaTsay2002}
Z.~Bai, B.~Miao, and J.~Tsay.
\newblock Convergence rates of the spectral distributions of large {W}igner
  matrices.
\newblock {\em Int. Math. J.}, 1(1):65--90, 2002.

\bibitem{BaiSilv2010}
Z.~Bai and J.~Silverstein.
\newblock {\em Spectral analysis of large dimensional random matrices}.
\newblock Springer, New York, second edition, 2010.

\bibitem{BobGotTikh2010}
S.~Bobkov, F.~G{\"o}tze, and A.~Tikhomirov.
\newblock On concentration of empirical measures and convergence to the
  semi-circle law.
\newblock {\em J. Theoret. Probab.}, 23(3):792--823, 2010.

\bibitem{Schlein2014}
C.~Cacciapuoti, A.~Maltsev, and B.~Schlein.
\newblock Bounds for the stieltjes transform and the density of states of
  wigner matrices.
\newblock {\em Probability Theory and Related Fields}, 163(1):1--59, 2015.

\bibitem{ErdKnowYauYin2012}
L.~Erd{\H{o}}s, A.~Knowles, H.-T. Yau, and J.~Yin.
\newblock Spectral statistics of {E}rd{\H o}s-{R}\'enyi {G}raphs {II}:
  {E}igenvalue spacing and the extreme eigenvalues.
\newblock {\em Comm. Math. Phys.}, 314(3):587--640, 2012.

\bibitem{ErdKnowYauYin2013}
L.~Erd{\H{o}}s, A.~Knowles, H.-T. Yau, and J.~Yin.
\newblock The local semicircle law for a general class of random matrices.
\newblock {\em Electron. J. Probab.}, 18:no. 59, 58, 2013.

\bibitem{ErdKnowYauYin2013a}
L.~Erd{\H{o}}s, A.~Knowles, H.-T. Yau, and J.~Yin.
\newblock Spectral statistics of {E}rd{\H o}s-{R}\'enyi graphs {I}: {L}ocal
  semicircle law.
\newblock {\em Ann. Probab.}, 41(3B):2279--2375, 2013.

\bibitem{ErdosSchleinYau2009b}
L.~Erd{\H{o}}s, B.~Schlein, and H.-T. Yau.
\newblock Local semicircle law and complete delocalization for {W}igner random
  matrices.
\newblock {\em Comm. Math. Phys.}, 287(2):641--655, 2009.

\bibitem{ErdosSchleinYau2009}
L.~Erd{\H{o}}s, B.~Schlein, and H.-T. Yau.
\newblock Semicircle law on short scales and delocalization of eigenvectors for
  {W}igner random matrices.
\newblock {\em Ann. Probab.}, 37(3):815--852, 2009.

\bibitem{ErdosSchleinYau2010}
L.~Erd{\H{o}}s, B.~Schlein, and H.-T. Yau.
\newblock Wegner estimate and level repulsion for {W}igner random matrices.
\newblock {\em Int. Math. Res. Not. IMRN}, (3):436--479, 2010.

\bibitem{Girko1985uspexi}
V.~Girko.
\newblock Spectral theory of random matrices.
\newblock {\em Uspekhi Mat. Nauk}, 40(1(241)):67--106, 256, 1985.

\bibitem{Girko1998}
V.~Girko.
\newblock Convergence rate of the expected spectral functions of symmetric
  random matrices is equal to {$O(n^{-1/2})$}.
\newblock {\em Random Oper. Stochastic Equations}, 6(4):359--408, 1998.

\bibitem{Girko2002}
V.~Girko.
\newblock Extended proof of the statement: convergence rate of the expected
  spectral functions of symmetric random matrices {$\Xi_n$} is equal to
  {$O(n^{-1/2})$} and the method of critical steepest descent.
\newblock {\em Random Oper. Stochastic Equations}, 10(3):253--300, 2002.

\bibitem{GotzeNauTikh2015a}
F.~G{\"o}tze, A.~Naumov, and A.~Tikhomirov.
\newblock Local semicircle law under moment condtions. {P}art {I}: The
  {S}tieltjes transfrom.
\newblock {\em arXiv:1510.07350}.

\bibitem{GotzeNauTikh2015b}
F.~G{\"o}tze, A.~Naumov, and A.~Tikhomirov.
\newblock Local semicircle law under moment condtions. {P}art {II}:
  Localization and delocalization.
\newblock {\em arXiv:1511.00862}.

\bibitem{GotNauTikh2012prob}
F.~G{\"o}tze, A.~Naumov, and A.~Tikhomirov.
\newblock Limit theorems for two classes of random matrices with dependent
  entries.
\newblock {\em Theory Probab. Appl.}, 59 (114)(1):23--39, 2015.

\bibitem{GotzeTikh2011rateofconv}
F.~G{\"o}tze and A.~Tikhomirov.
\newblock On the rate of convergence to the semi-circular law.
\newblock {\em arXiv:1109.0611}.

\bibitem{GotTikh2015weak}
F.~G{\"o}tze and A.~Tikhomirov.
\newblock Optimal bounds for convergence of expected spectral distributions to
  the semi-circular law for the $4+\varepsilon$ moment ensemble.
\newblock {\em arXiv:1511.03692}.

\bibitem{GotzeTikh2014rateofconv}
F.~G{\"o}tze and A.~Tikhomirov.
\newblock Rate of convergence of the empirical spectral distribution function
  to the semi-circular law.
\newblock {\em arXiv:1407.2780}.

\bibitem{GotzTikh2002}
F.~G{\"o}tze and A.~Tikhomirov.
\newblock Rate of convergence to the semi-circular law for the {G}aussian
  unitary ensemble.
\newblock {\em Teor. Veroyatnost. i Primenen.}, 47(2):381--387, 2002.

\bibitem{GotTikh2003}
F.~G{\"o}tze and A.~Tikhomirov.
\newblock Rate of convergence to the semi-circular law.
\newblock {\em Probab. Theory Related Fields}, 127(2):228--276, 2003.

\bibitem{GotzTikh2005}
F.~G{\"o}tze and A.~Tikhomirov.
\newblock The rate of convergence for spectra of {GUE} and {LUE} matrix
  ensembles.
\newblock {\em Cent. Eur. J. Math.}, 3(4):666--704 (electronic), 2005.

\bibitem{GotTikh2006}
F.~G{\"o}tze and A.~Tikhomirov.
\newblock Limit theorems for spectra of random matrices with martingale
  structure.
\newblock {\em Teor. Veroyatn. Primen.}, 51(1):171--192, 2006.

\bibitem{GotTikh2015}
F.~G{\"o}tze and A.~Tikhomirov.
\newblock Optimal bounds for convergence of expected spectral distributions to
  the semi-circular law.
\newblock {\em Probability Theory and Related Fields}, pages 1--71, 2015.

\bibitem{Gustavsson2005}
J.~Gustavsson.
\newblock Gaussian fluctuations of eigenvalues in the {GUE}.
\newblock {\em Ann. Inst. H. Poincar\'e Probab. Statist.}, 41(2):151--178,
  2005.

\bibitem{HansonWright}
D.~Hanson and F.~Wright.
\newblock A bound on tail probabilities for quadratic forms in independent
  random variables.
\newblock {\em Ann. Math. Statist.}, 42:1079--1083, 1971.

\bibitem{LeeYin2014}
J.~Lee and J.~Yin.
\newblock A necessary and sufficient condition for edge universality of
  {W}igner matrices.
\newblock {\em Duke Math. J.}, 163(1):117--173, 2014.

\bibitem{Naumov2013gaussiancase}
A.~Naumov.
\newblock Limit theorems for two classes of random matrices with {G}aussian
  elements.
\newblock {\em Journal of Mathematical Sciences}, 204(1):140--147, 2014.

\bibitem{Pastur1973}
L.~Pastur.
\newblock Spectra of random selfadjoint operators.
\newblock {\em Uspehi Mat. Nauk}, 28(1(169)):3--64, 1973.

\bibitem{RudelsonVershynin2013}
M.~Rudelson and R.~Vershynin.
\newblock Hanson-{W}right inequality and sub-{G}aussian concentration.
\newblock {\em Electron. Commun. Probab.}, 18:no. 82, 9, 2013.

\bibitem{Tao2012}
T.~Tao.
\newblock {\em Topics in random matrix theory}, volume 132 of {\em Graduate
  Studies in Mathematics}.
\newblock American Mathematical Society, 2012.

\bibitem{TaoVu2013}
T.~Tao and V.~Vu.
\newblock Random matrices: sharp concentration of eigenvalues.
\newblock {\em Random Matrices Theory Appl.}, 2(3):1350007, 31, 2013.

\bibitem{TimushevTikhKhol2007}
D.~Timushev, A.~Tikhomirov, and A.~Kholopov.
\newblock On the accuracy of the approximation of the {GOE} spectrum by the
  semi-circular law.
\newblock {\em Teor. Veroyatn. Primen.}, 52(1):180--185, 2007.

\bibitem{Wigner1955}
E.~Wigner.
\newblock Characteristic vectors of bordered matrices with infinite dimensions.
\newblock {\em Ann. of Math. (2)}, 62:548--564, 1955.

\end{thebibliography}

\end{document}